\documentclass[11pt]{article}

\usepackage[a4paper,margin=1in]{geometry}
\usepackage{amsmath,amssymb,amsthm,mathtools}
\usepackage{enumitem}
\usepackage{array}
\usepackage{cite}
\usepackage{tabularx}
\usepackage{ragged2e}
\usepackage{hyperref}

\hypersetup{
  colorlinks=true,
  linkcolor=blue,
  citecolor=blue,
  urlcolor=blue
}

\numberwithin{equation}{section}

\newtheorem{theorem}{Theorem}[section]
\newtheorem{proposition}[theorem]{Proposition}
\newtheorem{lemma}[theorem]{Lemma}
\newtheorem{corollary}[theorem]{Corollary}
\newtheorem{definition}[theorem]{Definition}
\newtheorem{remark}[theorem]{Remark}
\newtheorem{maintheorem}{Theorem}

\newcommand{\T}{\mathbb T}
\newcommand{\R}{\mathbb R}
\newcommand{\Id}{I}
\newcommand{\Spp}{\mathbb S_{++}}
\newcommand{\Sym}{\mathbb S}
\newcommand{\tr}{\operatorname{tr}}
\newcommand{\divv}{\operatorname{div}}

\newcommand{\Log}{\operatorname{Log}}
\newcommand{\dd}{\,dx}

\newcommand{\calE}{\mathcal E}
\newcommand{\calM}{\mathcal M}
\newcommand{\calN}{\mathcal N}

\newcommand{\calA}{\mathcal A}
\newcommand{\calD}{\mathcal D}
\newcommand{\calK}{\mathcal K}

\newcommand{\eps}{\varepsilon}
\newcommand{\avg}[1]{\left\langle #1\right\rangle}
\newcommand{\norm}[1]{\left\lVert #1\right\rVert}
\newcolumntype{Y}{>{\RaggedRight\arraybackslash}X}
\title{Residual-Work Compatibility Criteria and Defect-Measure Compactness\\
for Positive-Cone Viscoelastic Reynolds States}

\author{Sai Peng\\
School of Mathematics and Computational Science, Xiangtan University\\
\texttt{pscfd@xtu.edu.cn}}
\date{}

\begin{document}
\maketitle

\begin{abstract}
We prove a residual-work compatibility theory for positive-cone viscoelastic
Reynolds states.  For Oldroyd--B, the entropy cancellation removes
incompressible transport and upper-convected stretching against polymeric stress
work.  After quotienting pressure tensors and spatial means, positive
pressure-free work can be paid only by the conformation residual through the
entropy-dual lever $G=I-A^{-1}$.  This identifies the closed residual-work cone
\[
  P+\frac{\alpha}{2}\int_{\mathbb T^d}G:S\,dx\le0,
\]
exact windowed tests, and the least-cost Hilbert-space repair; constrained
closures pay through the projected lever.  Strong residual-data limits preserve
the cone, whereas weak--weak limits are closed only after retaining the
product-defect measure carried by $G:S$, and this augmented topology is sharp
via localized residual packets.  For FENE-P the entropy-dual lever
$G_b(C)=\frac{b-d}{b-\operatorname{tr}C}I-C^{-1}$ makes the
finite-extensibility boundary a genuine residual-work boundary.
\end{abstract}



\noindent\textbf{Short title.}
Residual-Work Compatibility.

\noindent\textbf{Correspondence.}
Sai Peng, School of Mathematics and Computational Science, Xiangtan
University, Xiangtan, China. Email: \texttt{pscfd@xtu.edu.cn}.

\section{Introduction}

The stress-diffusion-free Oldroyd--B and FENE-P systems contain a geometric
constraint that is absent from ordinary Reynolds-stress formulations: the
conformation tensor must remain in a positive cone, and in FENE-P it must also
stay below the finite-extensibility boundary.  In exact smooth solutions this
geometry is encoded by the entropy.  In relaxed descriptions, reduced models,
or coarse-grained Reynolds states, it imposes an additional restriction on the
defects: positive pressure-free work produced by the momentum residual must be
paid through the conformation residual in a direction seen by the entropy-dual
lever.

This paper isolates that residual-level restriction.  Its object is residual
admissibility rather than strong-solution continuation.  The pressure-quotient
and active-deviatoric cancellation mechanisms used in endpoint continuation
for smooth two-dimensional solutions \cite{PengEndpoint2026} belong to that
separate theory.  The present paper addresses the following question:

\begin{quote}
\emph{Given a smooth positive-cone Reynolds profile and its residuals, when is
its pressure-free work compatible with the Oldroyd--B or FENE-P
energy--entropy inequality?}
\end{quote}

The answer is a sharp signed work law.  For Oldroyd--B the entropy density is
\[
  \Phi(A)=\operatorname{tr}A-\log\det A-d,
\]
and the associated entropy-dual lever is
\[
  G(A)=D_A\Phi(A)=I-A^{-1}.
\]
After pressure and spatial-mean modes have been removed from the momentum
defect, admissibility forces the positive work to be balanced by the
conformation residual through the pairing \(-\int G:\mathcal S\).  This is not
an energy estimate with slack: the corresponding Hilbert-space cost is exact,
and equality identifies the unique aligned paying direction.

The theory is designed for residual traces produced by filtering,
coarse-graining, closure modeling, and a posteriori numerical reconstruction.
Given resolved fields and residuals, one forms the pressure-free work \(P\),
the entropy lever \(G\), and the conformation residual \(S\).  The criterion
then answers a concrete question: can this residual trace satisfy the
viscoelastic energy--entropy ledger after pressure and mean gauges have been
removed?  When the answer is no, the same theorem returns the exact
Hilbert-space repair and, for restricted closures, the projected direction
through which the closure must pay the work.

This point of view gives an operational bridge between the mathematics and
computations.  In a DNS-filtered or coarse-grid data set, \(P,G,S\) are computed
from the numerical residuals.  The finite-thickness shear layer in
Proposition~\ref{prop:concrete-layer} is used for a different purpose: it is a
calibration benchmark for the diagnostic, with explicit slopes for the work,
lever, residual, and violation ratio.  It verifies that the residual-work
postprocessor is measuring the right quantities before those quantities are
applied to a genuinely filtered simulation or closure data set.

The proof has two layers.  The PDE layer is the Oldroyd--B/FENE-P
energy--entropy cancellation: incompressible transport drops out of the entropy,
upper-convected stretching cancels polymeric stress work, and the remaining
conformation defect is paired with the entropy-dual lever.  This identifies the
unique channel that can pay pressure-free viscoelastic residual work.  The
geometric layer then treats the resulting signed budget as a Hilbert
half-space, proves the sharp metric projection, and determines the compactness
topology needed to preserve the product \(G:S\).  The final weak-solution
subsection formulates the corresponding packet-lifting problem as an explicit
open problem.

\subsection*{Relation to existing work and to the companion continuation theory}
The Oldroyd--B model originates in Oldroyd's constitutive theory and is a
standard differential model for dilute polymeric fluids; see also the classical
polymeric-fluid texts and analyses
\cite{Oldroyd,Bird,Renardy,GuillopeSaut,LionsMasmoudi}.
FENE-P introduces the finite-extensibility boundary and the Peterlin spring
factor; the role of this boundary is familiar from both numerical and analytic
studies of finitely extensible models \cite{Bird,ChilcottRallison1988,MasmoudiFENE}.
The log-conformation representation was introduced in numerical rheology to
preserve positivity at high Weissenberg number
\cite{FattalKupferman,FattalKupferman2005}; here it serves only as a coordinate
system for positive-cone Reynolds states.

The present paper is also distinct from direct continuation criteria for
stress-diffusion-free viscoelastic flows.  Such criteria monitor strong
solutions through velocity, stress, vorticity, or endpoint Besov quantities
\cite{BealeKatoMajda,CheminMasmoudi,LeiMasmoudiZhou}.  Stress diffusion and
small-data mechanisms lead to different regularity theories
\cite{BarrettBoyaval,ConstantinKliegl,ElgindiRousset,TuWangWen}.  A companion
manuscript by the author \cite{PengEndpoint2026} develops a pressure-quotient
and active-deviatoric endpoint continuation theory for smooth two-dimensional
Oldroyd--B and FENE-P solutions.  There the central problem is whether a smooth
solution can be continued once a velocity clock and the conformation geometry
are controlled.  Here the central problem is instead whether a prescribed
Reynolds-level residual profile is compatible with the energy--entropy
inequality after pressure and mean modes have been removed.  The two theories
therefore use overlapping geometric variables but answer different questions:
continuation controls genuine smooth solutions, while the present paper gives
sharp necessary and metric conditions for residual admissibility.  Finally, our
use of Reynolds defects is closer in spirit to admissibility and relaxation
questions for fluid equations \cite{DeLellisSzekelyhidi} and to the
oscillation/concentration bookkeeping of defect-measure compactness
\cite{DiPernaMajda,LionsIncompressible,FeireislBook} than to a classical
strong-solution continuation criterion.

\subsection*{Comparison with nearby theories}
The following table records the separation that is used throughout the paper.
It is included to make the scope of the present residual-data theorem explicit.

\begin{center}
\small
\renewcommand{\arraystretch}{1.12}
\begin{tabularx}{\textwidth}{YYY}
\hline
\textbf{Smooth continuation} & \textbf{Unconstrained Reynolds defects} &
\textbf{Positive-cone residual-work theory} \\
\hline
Exact smooth solutions are monitored by endpoint, vorticity, stress, or
conformation controls. & Momentum Reynolds stresses are relaxed variables;
conformation positivity and entropy payment are usually external constraints. &
A resolved or filtered positive-cone profile is tested through its signed
pressure-free residual work. \\
\hline
The output is a continuation criterion or a loss of continuation control. &
Weak limits may create stresses, concentrations, or product defects. & The
output is a computable ledger: cone membership, least-cost repair,
closure-subspace obstruction, and defect-measure compactification for weak
limits. \\
\hline
\end{tabularx}

\end{center}

\subsection*{Physical reading and numerical use}
The residual-work criterion has a direct physical interpretation.  The
momentum residual may inject positive pressure-free power.  Entropy
admissibility says that this power cannot disappear into pressure or mean
modes; it must be paid by a negative projection of the conformation residual
onto the entropy-dual lever.  The principal quantities used throughout the
paper may be read as follows.

\begin{center}
\small
\renewcommand{\arraystretch}{1.18}
\begin{tabularx}{\textwidth}{>{\RaggedRight\arraybackslash}p{0.13\textwidth}YY}
\hline
\textbf{Symbol} & \textbf{Residual-work role} & \textbf{Physical reading} \\
\hline
\(P\) & pressure-free momentum-residual work & the component of the coarse-grid
or Reynolds momentum defect that does work against the strain after pressure
and spatial-mean gauges have been removed \\
\(G=I-A^{-1}\) & entropy-dual lever & the conformation direction visible to
entropy; near equilibrium the lever is small, so paying positive work is
expensive \\
\(S\) & conformation residual & the unresolved conformation-equation defect;
only its projection onto \(-G\) pays positive pressure-free work \\
\(\eta_E\) & alignment efficiency & a measured number in \([0,1]\), not a
material parameter; it records how efficiently \(S\) points in the paying
direction over the window \(E\) \\
\(W_E,L_E,S_E\) & windowed work, lever, and residual sizes & the observable
channels in \(W_E\le(\alpha/2)\eta_E L_ES_E\); in certificates,
\(W_E^-\) is a certified lower bound and \(L_E^+,S_E^+\) are certified upper
bounds \\
\(\mu\) & weak product-defect measure & the hidden average of oscillatory
\(G:S\) products lost under weak--weak convergence \\
\hline
\end{tabularx}
\end{center}

\begin{figure}[h!]
\centering
\fbox{\begin{minipage}{0.88\textwidth}
\[
\begin{gathered}
\text{momentum residual }R
\ \longrightarrow\
P=\left[-\int R:D(u)\,dx\right]_{\rm pf}^{+}>0,\\[0.25em]
P+\frac{\alpha}{2}\int_{\T^d}G:S\,dx\le0,
\qquad
P>0\ \Longrightarrow\
-\int_{\T^d}G:S\,dx\ge\frac{2P}{\alpha},\\[0.25em]
\text{conformation residual }S
\text{ pays only through its projection onto }-G .
\end{gathered}
\]
If the last projection is absent, too small, or unresolved by the closure
subspace, the positive pressure-free work is inadmissible.
\end{minipage}}
\caption{Residual-work budget.  Pressure modes do not pay the ledger.  Positive
pressure-free work generated by the momentum residual must be balanced by the
conformation residual through the entropy-dual lever.}
\label{fig:residual-work-budget}
\end{figure}

\paragraph{A posteriori residual diagnostic for numerical data.}
The criterion can be used as a postprocessing test for a discrete Oldroyd--B,
FENE-P, or coarse-grained closure calculation.  As a postprocessor, it diagnoses
whether the residuals produced by a scheme or a closure are compatible with the
energy--entropy ledger.
\begin{enumerate}[label=\textbf{D\arabic*.},leftmargin=*]
\item On each stored time level, load or reconstruct
\((u_h,A_h,R_h,S_h)\).  Check positive-cone admissibility of \(A_h\), or use a
log-conformation representation to compute \(A_h^{-1}\) stably.
\item Compute the entropy lever: \(G_h=I-A_h^{-1}\) for Oldroyd--B, or
\(G_{b,h}=\frac{b-d}{b-\operatorname{tr}C_h}I-C_h^{-1}\) for FENE-P.
\item Remove the momentum gauges.  Equivalently, replace \(R_h\) by its
pressure-free representative, since isotropic pressure tensors and spatial
means do no work against \(D(u_h)\).  Then evaluate
\[
  P_h(t_n)=\left[-\int_{\T^d}R_h:D(u_h)\,dx\right]_{\rm pf},
  \qquad
  K_h(t_n)=\int_{\T^d}G_h:S_h\,dx .
\]
\item Form the signed defect gap
\[
  q_h(t_n)=P_h(t_n)+\frac{\alpha}{2}K_h(t_n).
\]
A resolved positive gap means that the discrete residual ledger is not
admissible at that time.
\item On a time window \(E\), compute
\[
  W_{E,h}=\int_E[P_h]_+\,dt,
  \qquad
  L_{E,h}=\|G_h\|_{L^2(E\times\T^d)},
  \qquad
  S_{E,h}=\|S_h\|_{L^2(E\times\T^d)},
\]
and either compute the alignment \(\eta_{E,h}\) or use a certified upper bound
\(\bar\eta_E\le1\).  The conservative exclusion test is
\[
  W_{E,h}^- > \frac{\alpha}{2}\bar\eta_E L_{E,h}^+S_{E,h}^+ +
  \delta_E^{\rm data},
\]
where the superscripts denote certified lower/upper bounds and
\(\delta_E^{\rm data}\) is the quadrature, projection, and residual-reconstruction
uncertainty.
\item If the data fail the test, compute the metric repair
\[
  S_h^{\rm rep}=S_h-
  \frac{2[q_h]_+}{\alpha\|G_h\|_{L^2}^2}G_h,
\]
or, for a restricted closure space \(\mathcal V_h\), replace \(G_h\) by
\(\Pi_{\mathcal V_h}G_h\).  The norm of this correction is a diagnostic measure
of how far the closure is from the admissible residual class.
\item If the gap is smaller than the certified data uncertainty, the diagnostic
returns an unresolved band.  In coarse-grid computations this separates
resolved violations from gaps below the residual-error floor.
\end{enumerate}
This workflow turns the abstract half-space theorem into a finite-dimensional
ledger test: compute the work, compute the lever, compute the paying projection,
and compare the resulting gap with the numerical uncertainty.

\subsection*{Four core contributions}
For readability the detailed statements below retain theorem-level navigation,
but the logical contribution is organized around four pillars.  First, the
Oldroyd--B entropy cancellation produces a pressure-free residual-work cone and
an exact Hilbert-space repair problem.  Second, the repair direction is shown to
be compatible with local positive-cone conformation dynamics under strict
slack.  Third, FENE-P is obtained through the correct entropy-dual lever, whose
blow-up identifies finite extensibility as a boundary residual-work phenomenon.
Fourth, strong residual-data closedness, topology sharpness, product-defect
compactification, and localized residual packets give the sharp compactness and
converse statements available at the residual-data level.

\subsection*{Core result 1: entropy-dual residual-work identity and cone}

\begin{maintheorem}[Defect-work identity and pressure-free reduction]
A positive-cone Oldroyd--B--Reynolds state is a smooth tuple
\((u,p,B,R,S)\), with \(A=e^B\), satisfying the momentum and conformation
equations with symmetric defects \(R\) and \(S\).  The free energy obeys
\[
\begin{aligned}
  \frac{d}{dt}\mathcal E(t)
  &+\nu\int_{\mathbb T^d}|\nabla u|^2\,dx
  +\frac{\alpha}{2\lambda}
   \int_{\mathbb T^d}\operatorname{tr}(A+A^{-1}-2I)\,dx \\
  &=
  -\int_{\mathbb T^d}R:D(u)\,dx
  +\frac{\alpha}{2}\int_{\mathbb T^d}(I-A^{-1}):S\,dx .
\end{aligned}
\]
Pressure tensors and spatial means do not contribute to the physical work.
Consequently the momentum defect has a canonical pressure-free representative
and a gauge-invariant work \(P(t)\).  This reduction is the starting point for
the sharp criterion.
\end{maintheorem}

\paragraph{Sharp cone and metric repair.}

The terminology in the following statement is made precise in
Sections~\ref{sec:notation}--\ref{sec:matching}.

\begin{maintheorem}[Residual-work compatibility criterion]
\label{thm:main-obstruction}
Let \((u,p,B)\) be a residual-matching profile on a time interval \(I\), and
let \(P(t)\) be its pressure-free work in
Definition~\ref{def:pressure-free-work}.  Set
\[
  A=e^B,\qquad G=I-A^{-1},\qquad
  K(t)=\int_{\mathbb T^d}G(t):\mathcal S[u,B](t)\,dx .
\]
If the associated residual-matching state is energy--entropy admissible, then
the pointwise signed-work inequality
\[
  P(t)+\frac{\alpha}{2}K(t)\le0
  \qquad\hbox{for a.e. }t\in I
\]
holds.  Equivalently, the pair \((P,K)\) belongs to the closed convex cone
\[
  \mathfrak C_E
  =
  \left\{(P,K):P+\frac{\alpha}{2}K\le0\hbox{ a.e. on }E\right\}
\]
on every measurable \(E\subset I\).  In windowed form,
\[
  W_E\le \frac{\alpha}{2}\eta_E L_ES_E,
\]
where
\[
  \eta_{\mathrm{al}}(t)=
  \begin{cases}
  \displaystyle
  \frac{\left[-K(t)\right]_+}
       {\|G(t)\|_{L^2}\|\mathcal S[u,B](t)\|_{L^2}},
  &\|G(t)\|_{L^2}\|\mathcal S[u,B](t)\|_{L^2}>0,\\[1.2em]
  0,&\text{otherwise},
  \end{cases}
  \qquad
  \eta_E=\operatorname*{ess\,sup}_{t\in E}\eta_{\mathrm{al}}(t).
\]
Thus the core criterion is the pointwise signed-work inequality; the windowed
bound is its Cauchy--Schwarz consequence and its computable a posteriori form.
\end{maintheorem}

\begin{proposition}[Optimal residual repair and Hilbert geometry]
\label{prop:intro-metric-repair}
At a fixed time, assume \(G\not\equiv0\).  Among all conformation residuals with
the same lever \(G\), the least \(L^2\)-size needed to pay positive work is
\[
  \frac{2[P]_+}{\alpha\|G\|_{L^2}},
\]
with minimizer
\[
  S_{\min}=
  -\frac{2P}{\alpha\|G\|_{L^2}^2}G
  \qquad(P>0).
\]
For a proposed residual \(S_0\), define
\[
  \mathfrak v=
  \left[P+\frac{\alpha}{2}\int_{\mathbb T^d}G:S_0\,dx\right]_+ .
\]
The nearest admissible residual is
\[
  \widehat S
  =
  S_0-\frac{2\mathfrak v}{\alpha\|G\|_{L^2}^2}G,
  \qquad
  \|\widehat S-S_0\|_{L^2}
  =
  \frac{2\mathfrak v}{\alpha\|G\|_{L^2}} .
\]
On time windows this pointwise projection gives the unique minimizer of the
global repair problem under the zero-lever and integrability conditions of
Theorem~\ref{thm:global-variational-projection}.  The only singular stratum is
\(G=0\) with positive work, where no finite conformation residual can pay the
ledger.
\end{proposition}

\paragraph{Closure-subspace repair.}

The metric repair in Proposition~\ref{prop:intro-metric-repair} assumes that the conformation residual can be
changed in every \(L^2\) direction.  Reduced models, finite-dimensional closures,
and numerical reconstructions often allow only a prescribed residual subspace.
The next result identifies the exact obstruction in that case.

\begin{maintheorem}[Constrained closure-subspace repair]
\label{thm:subspace-repair-main}
Fix a time and let \(\mathcal V\subset L^2(\mathbb T^d;\mathbb S^d)\) be a
closed linear subspace of admissible conformation-residual corrections.  For a
proposed residual \(S_0\), define
\[
  \mathfrak v=
  \left[P+\frac{\alpha}{2}\int_{\mathbb T^d}G:S_0\,dx\right]_+,
  \qquad H=\Pi_{\mathcal V}G .
\]
If \(H\not\equiv0\), the nearest residual of the form \(S_0+Z\),
\(Z\in\mathcal V\), which satisfies the admissibility inequality is
\[
  \widehat S_{\mathcal V}
  =
  S_0-
  \frac{2\mathfrak v}{\alpha\|H\|_{L^2}^2}H,
\]
and
\[
  \operatorname{dist}_{L^2}\left(S_0,(S_0+\mathcal V)\cap\mathcal A(P,G)\right)
  =
  \frac{2\mathfrak v}{\alpha\|H\|_{L^2}} .
\]
If \(H\equiv0\) and \(\mathfrak v>0\), no allowed correction in \(\mathcal V\)
can repair the positive work.  Thus a closure family can pay residual work only
through the component of the entropy-dual lever that it actually resolves.
\end{maintheorem}

\subsection*{Core result 2: local positive-cone compatibility of residual repair}

The optimal paying direction in Proposition~\ref{prop:intro-metric-repair} is an \(L^2\)-projection.  The next
result verifies that this direction is compatible with the positive-cone
conformation channel on short space-time slabs.

\begin{maintheorem}[Local positive-cone compatibility of the aligned repair]
\label{thm:local-corrector-main}
Let \(d\in\{2,3\}\), let \((u,p,B)\) be a smooth residual-matching profile on a
compact time interval \(J\), and set \(A=e^B\).  Assume that \(A\) lies in a
compact spectral envelope and that, on \(J\),
\[
  P(t)\ge p_0>0,
  \qquad
  \|G(t)\|_{L^2(\mathbb T^d)}\ge g_0>0,
  \qquad G=I-A^{-1}.
\]
For \(\gamma>0\), define the overpaying aligned residual
\[
  S_\gamma(t,x)
  =-
  \frac{2(1+\gamma)P(t)}{\alpha\|G(t)\|_{L^2}^2}
  G(t,x).
\]
For every \(t_0\in J\) there exist subintervals \(J_\tau\ni t_0\),
\(|J_\tau|\simeq\tau\), and smooth fields \(B_\tau\) such that
\(A_\tau=e^{B_\tau}\) remains in a slightly larger compact subset of
\(\mathbb S_{++}^d\),
\[
  \|A_\tau-A\|_{L^\infty(J_\tau;C^k_x)}\le C_k\tau
  \qquad\text{for every fixed }k,
\]
and
\[
  \|\mathcal S[u,B_\tau]-S_\gamma\|_{L^\infty(J_\tau;C^k_x)}\le C_k\tau
  \qquad\text{for every fixed }k.
\]
If \(P_\tau\) and \(G_\tau=I-A_\tau^{-1}\) are computed from \((u,B_\tau)\),
then, for sufficiently small \(\tau\),
\[
  P_\tau(t)+\frac{\alpha}{2}
  \int_{\mathbb T^d}G_\tau(t):\mathcal S[u,B_\tau](t)\,dx
  \le -\frac{\gamma p_0}{2},
  \qquad t\in J_\tau .
\]
Thus the aligned repair direction is locally compatible with the positive
cone and with the conformation residual-matching equation, up to a short-time
error.
\end{maintheorem}

\paragraph{Strict local sufficiency.}

The same construction gives a local converse in the conformation channel.
Beyond the aligned special case, any smooth target residual with strict slack
can be matched to leading order while preserving positivity.

\begin{maintheorem}[Strict local conformation-channel sufficiency]
\label{thm:strict-local-sufficiency}
Let \(d\in\{2,3\}\), let \((u,p,B)\) be a smooth residual-matching profile on a
compact time interval \(J\), and set \(A=e^B\), \(G=I-A^{-1}\).  Assume that
\(A\) lies in a compact spectral envelope.  Let \(T(t,x)\) be a smooth symmetric
conformation-residual target such that, for some \(\sigma>0\),
\[
  P(t)+\frac{\alpha}{2}\int_{\mathbb T^d}G(t):T(t)\,dx
  \le -\sigma,
  \qquad t\in J .
\]
For every \(t_0\in J\) there exist subintervals \(J_\tau\ni t_0\),
\(|J_\tau|\simeq \tau\), and smooth fields \(B_\tau\) such that
\(A_\tau=e^{B_\tau}\) remains in a slightly larger compact subset of
\(\mathbb S_{++}^d\),
\[
  \|A_\tau-A\|_{L^\infty(J_\tau;C^k_x)}\le C_k\tau,
  \qquad
  \|\mathcal S[u,B_\tau]-T\|_{L^\infty(J_\tau;C^k_x)}\le C_k\tau
\]
for every fixed \(k\), and the perturbed residual balance is strictly
admissible:
\[
  P_\tau(t)+\frac{\alpha}{2}
  \int_{\mathbb T^d}G_\tau(t):\mathcal S[u,B_\tau](t)\,dx
  \le -\frac{\sigma}{2},
  \qquad t\in J_\tau .
\]
Here \(G_\tau=I-A_\tau^{-1}\) and \(P_\tau\) is the canonical pressure-free work
computed from \((u,B_\tau)\).  Thus strict residual-work slack is locally sufficient for conformation-channel
residual matching inside the positive cone.
\end{maintheorem}

\subsection*{Core result 3: FENE-P finite-extensibility residual work}

The mechanism is not tied to the Hookean entropy.  For entropy-dual closures,
the lever is the derivative of the conformation entropy.  In particular, for
FENE-P on compact finite-extensibility envelopes, with \(b>d\),
\[
  f_b(C)=\frac{b-d}{b-\operatorname{tr}C},\qquad
  G_b(C)=f_b(C)I-C^{-1}.
\]
The finite-extensibility boundary is not a notational change: the isotropic
part of \(G_b\) becomes singular as \(\operatorname{tr}C\uparrow b\), and the
stability constants of the residual geometry degenerate at the boundary.

\begin{maintheorem}[Entropy-dual/FENE-P residual-work criterion]
\label{thm:intro-fenep}
For smooth FENE-P residual-matching profiles whose conformation tensors remain
in a compact finite-extensibility envelope, energy--entropy admissibility
implies, on every measurable time window \(E\),
\[
  W_E^{\rm FENE}
  \le
  \frac{\alpha}{2}\eta_E^{\rm FENE}
  L_E^{\rm FENE}S_E^{\rm FENE},
\]
where \(L_E^{\rm FENE}=\|G_b(C)\|_{L^2(E\times\mathbb T^d)}\),
\(S_E^{\rm FENE}\) is the conformation-residual size, and
\(\eta_E^{\rm FENE}\) is the corresponding negative-alignment coefficient.
At each fixed time the least residual that pays positive pressure-free work is
again the aligned entropy-dual residual, with \(G\) replaced by \(G_b(C)\).
Moreover, Proposition~\ref{prop:fenep-boundary-geometry} identifies the new
finite-extensibility boundary regime: on lower spectral envelopes,
\(\norm{G_b(C)}\to\infty\) as \(\operatorname{tr}C\uparrow b\), and the
derivative of the lever contains the singular factor
\((b-\operatorname{tr}C)^{-2}\).
\end{maintheorem}

\subsection*{Core result 4: residual-data closedness, topology sharpness, and compactification}

The residual-work budget is not merely a formal inequality attached to one
chosen representative.  It is closed under the natural strong convergence of the
three quantities that enter the budget: pressure-free work, entropy lever, and
conformation residual.

\begin{maintheorem}[Closed strong residual-data realization obstruction]
\label{thm:intro-closed-realization}
Let \((u,p,B)\) be a smooth residual-matching profile on a time window \(E\),
with data \(P,G,\mathcal S\).  Suppose that it is strongly residual-data
realizable on \(E\) by energy--entropy admissible positive-cone Reynolds states,
in the sense of Definition~\ref{def:strong-residual-realization}.  Then
\[
  P(t)+\frac{\alpha}{2}\int_{\mathbb T^d}G(t):\mathcal S(t)\,dx\le0
  \qquad\hbox{for a.e. }t\in E.
\]
Consequently the sharp windowed residual-work estimate holds on every
measurable \(F\subset E\).  If a profile violates this signed-work inequality on
a set of positive measure, or violates the windowed budget on any subwindow, it
admits no such strong residual-data realization.  In particular, the
finite-thickness shear-layer profiles constructed in
Proposition~\ref{prop:concrete-layer} are not strongly residual-data realizable
for all sufficiently small thicknesses.
\end{maintheorem}

\paragraph{Quantitative separation.}

The obstruction has a quantitative form.  For residual data \(P,G,S\) on a time
window \(E\), define the defect gap
\begin{maintheorem}[Quantitative residual-data separation]
\label{thm:intro-quantitative-gap}
\[
  \mathfrak D_E(P,G,S)
  =
  \int_E
  \left[
  P(t)+\frac{\alpha}{2}
  \int_{\mathbb T^d}G(t):S(t)\,dx
  \right]_+dt .
\]
If \((\widetilde P,\widetilde G,\widetilde S)\) is admissible, then
\[
\begin{aligned}
  \mathfrak D_E(P,G,S)
  &\le
  \|P-\widetilde P\|_{L^1(E)}
  +\frac{\alpha}{2}
   \|G-\widetilde G\|_{L^2(E\times\mathbb T^d)}
   \|S\|_{L^2(E\times\mathbb T^d)}\\
  &\quad
  +\frac{\alpha}{2}
   \|\widetilde G\|_{L^2(E\times\mathbb T^d)}
   \|S-\widetilde S\|_{L^2(E\times\mathbb T^d)} .
\end{aligned}
\]
Thus a positive defect gap gives an explicit lower bound on the residual-data
distance from the admissible class.  For the finite-thickness shear layers,
\(\mathfrak D_E\gtrsim\varepsilon^{-(2\beta+1)}\).
\end{maintheorem}

The practical diagnostic need not rely on an exact evaluation of
\(\eta_E\).  The signed defect gap \(\mathfrak D_E\) is computed directly from
the residual data and is Lipschitz stable under \(L^1\times L^2\times L^2\)
data errors.  If only bulk estimates are available, any certified upper bound
\(\eta_E\le\bar\eta_E\), including the universal bound \(\bar\eta_E=1\), gives
a conservative exclusion test.  Thus the sharp constants describe the limiting
threshold, while applications use error bars and conservative upper bounds.

\paragraph{Topology sharpness.}

The strong topology in Theorem~G is not only a convenient proof device.  The
product \(G:S\) is the quantity that pays pressure-free work, and this product
is not determined by weak limits of \(G\) and \(S\).

\begin{maintheorem}[Sharpness of residual-data closure]
\label{thm:intro-weak-sharpness}
The admissible residual class remains closed if one of the two factors
\(G\) or \(S\) converges strongly and the other converges weakly in
\(L^2(E\times\mathbb T^d)\), together with \(P_n\to P\) in \(L^1(E)\).
However, weak convergence of both factors is insufficient.  There exist smooth
oscillatory residual data \((P_n,G_n,S_n)\) with \(P_n\to p_0>0\) strongly in
\(L^1(E)\) and \(G_n\rightharpoonup0\), \(S_n\rightharpoonup0\) weakly in
\(L^2(E\times\mathbb T^d)\), such that every triple \((P_n,G_n,S_n)\) is
admissible while the weak limit \((p_0,0,0)\) violates the signed-work
inequality.  Thus no weak-weak closure theorem can be formulated using only
the weak limits of \(P,G,S\).
\end{maintheorem}

\paragraph{Defect-measure compactification.}

The preceding failure has a precise closure mechanism.  A weak residual limit is
not described only by \((P,G,S)\); it must also retain the product defect carried
by \(G_n:S_n\).

\begin{maintheorem}[Defect-measure compactification of weak residual limits]
\label{thm:intro-defect-compactification}
Let \(P_n\to P\) in \(L^1(E)\), let \(G_n\rightharpoonup G\) and
\(S_n\rightharpoonup S\) weakly in \(L^2(E\times\mathbb T^d)\), and suppose that
the product measures
\[
  \left(\int_{\mathbb T^d}G_n(t):S_n(t)\,dx\right)dt
\]
converge weakly-* to
\[
  \left(\int_{\mathbb T^d}G(t):S(t)\,dx\right)dt+\mu
\]
for a finite signed measure \(\mu\) on \(E\).  If all
\((P_n,G_n,S_n)\) are admissible, then the augmented limit satisfies
\[
  P\,dt+\frac{\alpha}{2}
  \left(\int_{\mathbb T^d}G:S\,dx\right)dt
  +\frac{\alpha}{2}\mu\le0
\]
as a signed measure on \(E\).  Conversely, every smooth density \(m(t)\,dt\)
can occur as such a product defect for weakly null smooth oscillatory pairs.
Thus two sequences with the same marginal weak limits may carry different
product defects; the augmented variable is not a bookkeeping choice but the
missing component of the weak residual topology.
\end{maintheorem}

\paragraph{Strict augmented residual-data converse.}

The augmented budget is not only a closure condition.  At the residual-data
level it is also sufficient, provided one works with strict slack.

\begin{maintheorem}[Strict converse for positive-cone residual data]
\label{thm:intro-strict-augmented-converse}
Let \(E\) be a compact time interval.  Let \(P,G,S\) be smooth residual data,
let \(m\in C^\infty(E)\), assume that \(I-G(t,x)\ge\kappa I\) for some
\(\kappa>0\), and assume that, for some \(\sigma>0\),
\[
  P(t)+\frac{\alpha}{2}\int_{\mathbb T^d}G(t):S(t)\,dx
  +\frac{\alpha}{2}m(t)\le-\sigma
  \qquad\hbox{on }E .
\]
Then there are smooth admissible residual triples
\((P_n,G_n,S_n)\) such that \(P_n=P\), \(I-G_n\ge(\kappa/2)I\),
\[
  G_n\rightharpoonup G,\qquad S_n\rightharpoonup S
  \quad\hbox{weakly in }L^2(E\times\mathbb T^d),
\]
and
\[
  \left(\int_{\mathbb T^d}G_n(t):S_n(t)\,dx\right)dt
  \stackrel{*}{\rightharpoonup}
  \left(\int_{\mathbb T^d}G(t):S(t)\,dx\right)dt+m(t)\,dt .
\]
For all sufficiently large \(n\), the approximating triples satisfy the strict
pointwise admissibility inequality with margin \(\sigma/2\).  Thus the
defect-measure budget has a positive-cone converse realization at the
residual-data level.  Moreover, the oscillatory corrections \(G_n-G\) and
\(S_n-S\) may be supported in any prescribed non-empty open set
\(U\subset\mathbb T^d\), with arbitrarily small \(L^\infty\) lever amplitude
and explicit reciprocal \(L^2\) cost in the conformation residual; this
reciprocal cost is optimal up to constants under such localized small-lever
constraints.
\end{maintheorem}

\subsection*{Interface with weak-solution realization}

The main theorem list above stops at the residual-data converse.  Subsection~\ref{sec:weak-compactness}
then records the interface with genuine weak solutions.  Theorem~\ref{thm:weak-compactness-realization}
shows that any weak-solution realizing sequence whose defects can be identified
must satisfy the same augmented residual-work budget.  Theorem~\ref{thm:strong-error-lifting-certificate}
gives an explicit residual-error certificate for an admissible packet lift,
Theorem~\ref{thm:conditional-full-converse} states the corresponding converse
under such a certificate, and Theorem~\ref{thm:packet-ledger-obstruction} explains why a
stretching-packet cancellation does not automatically remove the weak product
defect.  These statements locate the PDE realization problem after the
residual theorem; the residual theorem itself is the closed contribution of the
paper.

\subsection*{Organization}
Section~\ref{sec:notation} fixes the residual-work notation and the residual
data spaces, strong topology, augmented topology, and residual-work cone.
Section~\ref{sec:states}
derives the defect-work identity.  Section~\ref{sec:matching} removes pressure
and mean gauges, identifies the budget obeyed by dynamically generated residual
traces, and proves the sharp residual cost, including the global variational
projection and constrained closure-subspace repair.  Section~\ref{sec:local-corrector}
proves the local positive-cone corrector and realizes strict metric repairs in
the conformation channel.  Section~\ref{sec:application}
gives structured-family tests and a finite-thickness obstruction.  Section~\ref{sec:fenep-residual}
extends the signed work law to entropy-dual closures and identifies the FENE-P
finite-extensibility boundary geometry.
Section~\ref{sec:closed-realization} proves closedness of the admissible
residual class, turns residual-work violations into strong non-realization
statements, identifies the product-defect compactification, proves the strict
positive-cone residual-data converse for augmented budgets, and gathers the
weak-solution compactness interface in Subsection~\ref{sec:weak-compactness}.
This keeps the weak compactness language next to the defect measure that closes
the residual ledger.  The final discussion and conclusion return to the
residual interpretation, numerical diagnostic meaning, and open lifting
problem.

\section{Residual-work notation in dimension \texorpdfstring{$d$}{d}}
\label{sec:notation}

In Sections~\ref{sec:states}--\ref{sec:closed-realization} we fix
\(d\in\{2,3\}\) and work on \(\T^d=(\R/2\pi\mathbb Z)^d\).  Matrix
contractions are Frobenius products, \(A:C=\tr(A^TC)\), and spatial averages
are denoted by
\[
  \avg{f}=|\T^d|^{-1}\int_{\T^d}f(x)\,dx .
\]
All function spaces in these sections are over \(\T^d\).  We write
\(H^s=H^s(\T^d)\) and \(L^2=L^2(\T^d)\), with componentwise norms for vector
and matrix fields.  Space-time norms over a measurable time window \(E\subset
[0,T]\) are written, for instance, as \(L^2(E\times\T^d)\).  The conformation
tensor is parametrized by
\[
  A=e^B,
  \qquad
  B(t,x)\in\Sym^d,
\]
so \(A(t,x)\in\Spp^d\) automatically.  The Hookean entropy and entropy-dual
lever are
\[
  \Phi_d(A)=\tr A-\log\det A-d,
  \qquad
  G(A)=D_A\Phi_d(A)=\Id-A^{-1}.
\]

For a smooth profile \((u,p,B)\), the momentum and conformation residuals are
\(\calM[u,p,B]\) and \(\mathcal S[u,B]\).  The momentum residual is always
understood modulo spatial means and pressure tensors.  A pressure-free
symmetric anti-divergence \(R_{\calN}\) is paired with \(D(u)\) through
\[
  P(t)=\int_{\T^d}-R_{\calN}(t):D(u(t))\,dx .
\]
Lemma~\ref{lem:gauge} shows that this number is independent of the
pressure-free representative.  On a time window \(E\), the residual-work
diagnostics are
\[
  W_E=\int_E[P(t)]_+\,dt,
  \qquad
  L_E=\norm{G}_{L^2(E\times\T^d)},
  \qquad
  S_E=\norm{\mathcal S[u,B]}_{L^2(E\times\T^d)},
\]
together with the following alignment coefficient.

\begin{definition}[Windowed residual-work data]
\label{def:windowed-data}
For a smooth residual-matching profile and a measurable time window \(E\), set
\[
  \eta_{\rm al}(t)=
  \begin{cases}
  \displaystyle
  \frac{\left[-\int_{\T^d}G(t):\mathcal S[u,B](t)\,dx\right]_+}
       {\|G(t)\|_{L^2}\|\mathcal S[u,B](t)\|_{L^2}},
  & \|G(t)\|_{L^2}\|\mathcal S[u,B](t)\|_{L^2}>0,\\[1.2em]
  0, & \text{otherwise},
  \end{cases}
\]
and
\[
  \eta_E=\operatorname*{ess\,sup}_{t\in E}\eta_{\rm al}(t).
\]
The windowed residual-work data are
\[
  (W_E,L_E,S_E,\eta_E).
\]
The degenerate convention above means that positive work on a zero-lever or
zero-residual set is immediately inadmissible.
\end{definition}

These four quantities are the complete input for the windowed admissibility
test:
\[
\begin{array}{c|l|l}
\hbox{symbol} & \hbox{definition} & \hbox{role}\\ \hline
P(t) & \displaystyle\int_{\T^d}-R_{\calN}:D(u)\,dx
& \hbox{gauge-invariant pressure-free work}\\[0.5em]
W_E & \displaystyle\int_E[P(t)]_+\,dt
& \hbox{positive work to be paid}\\[0.5em]
L_E & \norm{G}_{L^2(E\times\T^d)}
& \hbox{entropy-lever size}\\[0.5em]
S_E & \norm{\mathcal S[u,B]}_{L^2(E\times\T^d)}
& \hbox{conformation-residual size}\\[0.5em]
\eta_E & \operatorname*{ess\,sup}_{t\in E}\eta_{\mathrm{al}}(t)
& \hbox{negative-alignment efficiency}
\end{array}
\]

The condition is windowed rather than a frozen spatial snapshot.  The scalar
\(P(t)\) is computed after the full time-dependent momentum residual has been
reduced modulo pressure and mean gauges, so the convective, viscous, polymeric,
and forcing contributions of a proposed evolution enter through the same
pressure-free work.  For disjoint time windows \(E_j\),
\[
  W_{\cup_j E_j}=\sum_j W_{E_j},
\]
and the admissibility budget must hold on each subwindow as well as on their
union.  Thus sustained positive pressure-free production cannot be hidden by a
favorable instantaneous picture at isolated times; it must be paid by the
time-integrated lever--residual pairing on every time slab.

\begin{definition}[Residual-data spaces and topologies]
\label{def:residual-data-space}
For a measurable time window \(E\), set
\[
  \mathfrak X_E
  =
  L^1(E)\times L^2(E\times\T^d;\Sym^d)
  \times L^2(E\times\T^d;\Sym^d).
\]
An element is written \(X=(P,G,S)\).  The strong residual-data topology on
\(\mathfrak X_E\) is the product topology
\[
  P_n\to P\hbox{ in }L^1(E),\qquad
  G_n\to G,\quad S_n\to S\hbox{ in }L^2(E\times\T^d).
\]
The one-strong residual topology means \(P_n\to P\) in \(L^1(E)\), one of
\(G_n,S_n\) converges strongly in \(L^2\), and the other converges weakly in
\(L^2\).  The augmented residual-data space is
\[
  \mathfrak X_E^{\rm aug}
  =
  \mathfrak X_E\times\mathcal M(E),
\]
where \(\mathcal M(E)\) is the space of finite signed Radon measures on \(E\)
with the weak-* topology.  We write
\[
  (P_n,G_n,S_n)\to(P,G,S,\mu)
\]
in the augmented topology if \(P_n\to P\) in \(L^1(E)\),
\[
  G_n\rightharpoonup G,\qquad S_n\rightharpoonup S
  \quad\hbox{weakly in }L^2(E\times\T^d),
\]
and
\[
  \left(\int_{\T^d}G_n(t):S_n(t)\,dx\right)dt
  \stackrel{*}{\rightharpoonup}
  \left(\int_{\T^d}G(t):S(t)\,dx\right)dt+\mu
  \quad\hbox{in }\mathcal M(E).
\]
\end{definition}

\begin{definition}[Residual-work cone and admissible fibers]
\label{def:residual-work-cone}
For \(P,K\in L^1(E)\), define the residual-work cone
\[
  \mathfrak C_E
  =
  \left\{(P,K)\in L^1(E)\times L^1(E):
  P+\frac{\alpha}{2}K\le0\hbox{ a.e. on }E\right\}.
\]
This is a closed convex cone in \(L^1(E)\times L^1(E)\).  For
\((P,G)\in L^1(E)\times L^2(E\times\T^d;\Sym^d)\), the admissible fiber over
\((P,G)\) is
\[
  \mathfrak A_E(P,G)
  =
  \left\{S\in L^2(E\times\T^d;\Sym^d):
  \left(P,\int_{\T^d}G:S\,dx\right)\in\mathfrak C_E\right\}.
\]
Equivalently,
\[
  P(t)+\frac{\alpha}{2}\int_{\T^d}G(t):S(t)\,dx\le0
  \quad\hbox{for a.e. }t\in E.
\]
Thus admissibility is linear and convex in the signed-work variables
\((P,K)\), and it is a closed Hilbert half-space in the conformation-residual
fiber once \(P\) and \(G\) are fixed.  The pullback in the raw variables
\((P,G,S)\) is not asserted to be convex because the map
\((G,S)\mapsto\int G:S\) is bilinear.
\end{definition}

\begin{proposition}[Geometry of the residual-work admissible set]
\label{prop:residual-cone-geometry}
The set \(\mathfrak C_E\) is a closed convex cone in
\(L^1(E)\times L^1(E)\).  Its norm interior in \(L^1\times L^1\) is empty, as
is usual for order cones in \(L^1\), so the useful geometry is an order
stratification.  Writing
\[
  M_{P,K}=P+\frac{\alpha}{2}K,
\]
the saturated stratum is \(\{M_{P,K}=0\}\), the strictly dissipative stratum is
\(\{M_{P,K}<0\}\), and the exterior violation is measured by the positive part
\([M_{P,K}]_+\).  In the signed-work coordinate,
\[
  \operatorname{dist}_{L^1}(M_{P,K},L^1_-(E))
  =
  \norm{[M_{P,K}]_+}_{L^1(E)},
\]
where \(L^1_-(E)=\{f\in L^1(E):f\le0\hbox{ a.e.}\}\).  For fixed \((P,G)\),
the fiber \(\mathfrak A_E(P,G)\) is closed and convex in
\(L^2(E\times\T^d;\Sym^d)\).
If \(G(t)\not\equiv0\) for a.e. \(t\), the fiber boundary is the affine
hyperplane
\[
  P(t)+\frac{\alpha}{2}\int_{\T^d}G(t):S(t)\,dx=0
  \quad\hbox{for the active times}.
\]
The singular stratum is the zero-lever set
\[
  Z_G=\{t\in E:\norm{G(t)}_{L^2_x}=0\}.
\]
On \(Z_G\), positive work cannot be repaired by any conformation residual,
whereas non-positive work imposes no conformation-residual constraint.
\end{proposition}

\begin{proof}
The map \((P,K)\mapsto P+(\alpha/2)K\) is continuous from
\(L^1(E)\times L^1(E)\) to \(L^1(E)\), and the cone of non-positive
\(L^1\)-functions is closed and convex.  This proves that
\(\mathfrak C_E\) is closed and convex; it is a cone because the defining
inequality is homogeneous under multiplication by non-negative scalars.  The
positive-part distance formula follows from the pointwise projection of an
\(L^1\)-function onto \(L^1_-(E)\).  For fixed \((P,G)\), the map
\[
  S\mapsto P+\frac{\alpha}{2}\int_{\T^d}G:S\,dx
\]
is affine and continuous from \(L^2(E\times\T^d)\) to \(L^1(E)\), by Cauchy's
inequality.  Hence its inverse image of the non-positive cone is closed and
convex.  The description of the zero-lever stratum is immediate from the fact
that the pairing with \(S\) vanishes when \(G(t)\equiv0\).
\end{proof}

\section{Positive-Cone Reynolds States}
\label{sec:states}

\begin{definition}[Positive-cone Reynolds state]
\label{def:positive-cone-reynolds-state}
Let \(B(t,x)\in\Sym^d\) and set \(A=e^B\).  A positive-cone
Oldroyd--B--Reynolds state is a tuple
\[
  (u,p,B,R,S)
\]
with \(\divv u=0\), \(R,S\in\Sym^d\), satisfying
\begin{align}
  \partial_t u+\divv(u\otimes u)-\nu\Delta u+\nabla p
  &=
  \alpha\divv(A-\Id)+\divv R, \label{eq:reynolds-momentum}\\
  \partial_t A+u\cdot\nabla A-\nabla u\,A-A(\nabla u)^T
  +\lambda^{-1}(A-\Id)
  &=
  S. \label{eq:reynolds-conf}
\end{align}
\end{definition}

\begin{definition}[Defect admissibility]
\label{def:defect-admissibility}
A smooth positive-cone Reynolds state is energy--entropy admissible if
\[
  \int_{\T^d}
  \left[
  -R:D(u)+\frac{\alpha}{2}(\Id-A^{-1}):S
  \right]\dd
  \le0
  \quad\hbox{for a.e. }t.
  \label{eq:admissible}
\]
\end{definition}

\begin{definition}[Free energy and entropy lever]
For \(A\in\Spp^d\), define
\[
  \Phi(A)=\tr A-\log\det A-d,
  \qquad
  G(A)=D_A\Phi(A)=\Id-A^{-1}.
\]
The free energy of a positive-cone state is
\[
  \calE(t)=
  \frac12\int_{\T^d}|u(t)|^2\,dx
  +
  \frac{\alpha}{2}\int_{\T^d}\Phi(A(t))\,dx .
\]
\end{definition}

\begin{lemma}[Lever bounds on spectral windows]
\label{lem:lever-bounds}
If the spectrum of \(A\) lies in \([m,M]\), \(0<m<M<\infty\), then
\[
  |G(A)|\le C_{m,M}|A-\Id|,
  \qquad
  |A-\Id|\le C_{m,M}|G(A)|.
\]
Near equilibrium, \(G(A)=A-\Id+O(|A-\Id|^2)\).
\end{lemma}

\begin{proof}
The map \(A\mapsto \Id-A^{-1}\) is smooth on the compact spectral set
\(\{m\Id\le A\le M\Id\}\).  Its derivative at \(\Id\) is the identity map on
symmetric matrices, which gives the local expansion.
\end{proof}

\begin{proposition}[Defect-work identity]
\label{prop:defect-work}
Every smooth positive-cone Reynolds state satisfies
\[
\begin{aligned}
  \frac{d}{dt}\calE(t)
  &+\nu\int_{\T^d}|\nabla u|^2\dd
  +\frac{\alpha}{2\lambda}
  \int_{\T^d}\tr(A+A^{-1}-2\Id)\dd\\
  &=
  -\int_{\T^d}R:D(u)\dd
  +\frac{\alpha}{2}\int_{\T^d}(\Id-A^{-1}):S\dd .
\end{aligned}
\]
\end{proposition}

\begin{proof}
Testing \eqref{eq:reynolds-momentum} by \(u\), integrating over the torus, and
using \(\divv u=0\), gives
\[
  \frac12\frac{d}{dt}\int_{\T^d}|u|^2\,dx
  +\nu\int_{\T^d}|\nabla u|^2\,dx
  =
  -\alpha\int_{\T^d}(A-\Id):\nabla u\,dx
  -\int_{\T^d}R:D(u)\,dx .
\]
The pressure and transport terms vanish, and \(R:D(u)=R:\nabla u\) because
\(R\) is symmetric.

For the elastic part, since \(D_A\Phi(A)=G=\Id-A^{-1}\),
\[
  \frac{d}{dt}\int_{\T^d}\Phi(A)\,dx
  =
  \int_{\T^d}G:\partial_t A\,dx .
\]
The transport term satisfies
\[
  \int_{\T^d}G:(u\cdot\nabla A)\,dx
  =
  \int_{\T^d}u\cdot\nabla\Phi(A)\,dx=0.
\]
Moreover \(G\) commutes with \(A\), and
\[
  G:(\nabla u\,A+A(\nabla u)^T)
  =
  2(A-\Id):\nabla u .
\]
Finally,
\[
  G:(A-\Id)=\tr(A+A^{-1}-2\Id).
\]
Pairing \eqref{eq:reynolds-conf} with \(\alpha G/2\) therefore yields
\[
  \frac{\alpha}{2}\frac{d}{dt}\int_{\T^d}\Phi(A)\,dx
  =
  \alpha\int_{\T^d}(A-\Id):\nabla u\,dx
  -\frac{\alpha}{2\lambda}
  \int_{\T^d}\tr(A+A^{-1}-2\Id)\,dx
  +\frac{\alpha}{2}\int_{\T^d}G:S\,dx .
\]
Adding the kinetic and elastic balances cancels the principal stress work and
gives the stated identity.
\end{proof}

\begin{remark}[Where the PDE structure enters]
\label{rem:pde-structure-input}
The identity above is the only point at which the special nonlinear structure
of Oldroyd--B is compressed into the residual theory.  The cancellation uses
three PDE facts: incompressible transport is entropy conservative, the
upper-convected stretching term produces exactly the elastic stress work, and
the relaxation term is non-negative in the trace-log entropy.  After these
facts have reduced the defect contribution to
\[
  -\int_{\T^d}R:D(u)\,dx+\frac{\alpha}{2}\int_{\T^d}G:S\,dx ,
\]
the remaining sharp-cost, projection, closedness, and defect-measure arguments
are functional-analytic.  This separation is deliberate: the paper studies the
residual consequence of the PDE energy--entropy structure, not a strong
regularity mechanism for the original evolution.
\end{remark}

\section{Residual Matching and Sharp Cost}
\label{sec:matching}

For smooth \(u,p,B\), define
\begin{align}
  \calM[u,p,B]
  &=
  \partial_t u+\divv(u\otimes u)-\nu\Delta u+\nabla p
  -\alpha\divv(A-\Id),\\
  \mathcal S[u,B]
  &=
  \partial_t A+u\cdot\nabla A-\nabla u\,A-A(\nabla u)^T
  +\lambda^{-1}(A-\Id).
\end{align}
The residual matching equations are
\[
  \divv R=\calM[u,p,B],\qquad S=\mathcal S[u,B].
\]

\begin{definition}[Residual-matching profile]
\label{def:residual-matching-profile}
A smooth triple \((u,p,B)\), with \(\divv u=0\) and \(A=e^B\), is called a
residual-matching profile if the mean of \(\calM[u,p,B]\) vanishes for a.e.
time and if \(S=\mathcal S[u,B]\) is used as the conformation residual.  A symmetric
tensor \(R\) is an admissible momentum representative for the profile if, for
some scalar pressure correction \(\pi\),
\[
  \divv R+\nabla\pi=\calM[u,p,B].
\]
Two representatives are identified if their difference is a divergence-free
symmetric tensor plus an isotropic tensor \(q\Id\), since the latter changes
only the pressure correction in the momentum equation.
\end{definition}

\begin{definition}[Pressure-free work of a profile]
\label{def:pressure-free-work}
For a residual-matching profile, set
\[
  \calN[u,B]=
  \partial_t u+\divv(u\otimes u)-\nu\Delta u-\alpha\divv(A-\Id).
\]
Let
\[
  \calN_0(t,x)=\calN(t,x)-\int_{\T^d}\calN(t,y)\,dy
\]
be the zero-mean part.  Let \(\mathfrak R_{\calN}(t)\) be the affine space of
all symmetric tensors \(R\) for which there exists a scalar \(\pi\) satisfying
\[
  \divv R+\nabla\pi=\calN_0 .
\]
We quotient this affine space by the linear subspace
\[
  \{Q+q\Id:\ Q\in L^2(\T^d;\Sym^d),\ \divv Q=0,\ q\in L^2(\T^d)\}.
\]
A pressure-free representative is any element of the resulting quotient class.
The pressure-free work is
\[
  P(t)=\int_{\T^d}-R_{\calN}(t):D(u(t))\,dx .
\]
By Lemma~\ref{lem:gauge}, this number is independent of the representative.
The removed mean is the harmless mean-velocity forcing mode and does no work
against periodic gradients.
\end{definition}

\begin{remark}[Mean mode]
On the torus, a constant vector field cannot be represented as the divergence
of a periodic tensor.  This is why the spatial mean of \(\calN\) is removed in
Definition~\ref{def:pressure-free-work}.  The residual-work obstruction is
concerned with deformation work, and \(D(u)\) only sees the gradient part of
the velocity.  Fixing the mean velocity therefore removes no admissibility
information relevant to the pressure-free work.
\end{remark}

\begin{proposition}[Symmetric pressure-free antidivergence]
\label{prop:antidiv}
If a vector field \(m\) on \(\T^d\) has zero mean, then there is a symmetric
tensor \(R_m\) with
\[
  \divv R_m=m,\qquad
  \norm{R_m}_{H^{s+1}}\le C_s\norm{m}_{H^s}.
\]
\end{proposition}

\begin{proof}
Let \(\Delta\phi=m\) with zero mean and set
\[
  (R_m)_{ij}=\partial_i\phi_j+\partial_j\phi_i-\delta_{ij}\divv\phi.
\]
Then \(R_m\) is symmetric and
\[
  (\divv R_m)_j
  =\Delta\phi_j+\partial_j\divv\phi-\partial_j\divv\phi
  =m_j .
\]
Elliptic estimates for the zero-mean Poisson equation give the displayed
\(H^{s+1}\) bound.
\end{proof}

\begin{lemma}[Gauge invariance of pressure-free work]
\label{lem:gauge}
Let \(R_1,R_2\in L^2(\T^d;\Sym^d)\) satisfy
\[
  \divv R_1=\divv R_2
\]
and assume \(R_1-R_2\) differs from a divergence-free symmetric tensor by an
isotropic pressure tensor.  Then, for every divergence-free \(u\),
\[
  \int_{\T^d}R_1:D(u)\,dx
  =
  \int_{\T^d}R_2:D(u)\,dx .
\]
\end{lemma}

\begin{proof}
Write \(R_1-R_2=Q+p\Id\), with \(\divv Q=0\).  Then
\[
  \int_{\T^d}Q:D(u)\,dx
  =
  -\int_{\T^d}(\divv Q)\cdot u\,dx=0,
\]
and
\[
  \int_{\T^d}p\Id:D(u)\,dx
  =
  \int_{\T^d}p\,\divv u\,dx=0.
\]
\end{proof}

\begin{proposition}[Canonical pressure-free reduction]
\label{prop:canonical-pressure-free}
For every residual-matching profile, the value of \(P(t)\) in
Definition~\ref{def:pressure-free-work} is independent of the chosen
pressure-free representative.  Moreover, the energy--entropy admissibility of
the profile depends only on \(P(t)\), \(G(A)\), and \(\mathcal S[u,B]\).
\end{proposition}

\begin{proof}
The independence of \(P(t)\) is Lemma~\ref{lem:gauge}.  Substituting
\(S=\mathcal S[u,B]\) into the defect-work identity shows that all representatives
with the same pressure-free work produce the same right-hand side in the
energy--entropy balance.  Therefore admissibility is determined by
\[
  P(t)+\frac{\alpha}{2}
  \int_{\T^d}G(t):\mathcal S[u,B](t)\,dx\le0 .
\]
\end{proof}

\begin{proposition}[Pressure-free residual work]
\label{prop:pressure-free}
For a residual-matching profile, the admissibility condition is
\[
  P(t)+\frac{\alpha}{2}
  \int_{\T^d}G(t):\mathcal S[u,B](t)\,dx\le0,
  \label{eq:pressure-free}
\]
for a.e. \(t\).
\end{proposition}

\begin{proof}
Adding an isotropic pressure tensor \(p\Id\) changes the divergence constraint
but not the work pairing, since \(p\Id:D(u)=p\divv u=0\).  The remaining
gauge freedom is removed by Lemma~\ref{lem:gauge}.  Equivalently, this is the
canonical pressure-free reduction in
Proposition~\ref{prop:canonical-pressure-free}.
\end{proof}

\begin{proposition}[Dynamically generated residual traces]
\label{prop:dynamical-residual-traces}
Let \((u,p,B,R,S)\) be an energy--entropy admissible positive-cone Reynolds
state on a time window \(E\), and let \(P\) be the pressure-free work of its
momentum residual.  Then the associated residual trace satisfies
\begin{equation}
  P(t)+\frac{\alpha}{2}\int_{\T^d}G(t):S(t)\,dx\le0
  \qquad\hbox{for a.e. }t\in E .
  \label{eq:dynamical-trace-budget}
\end{equation}
In particular, an exact smooth Oldroyd--B solution has \(R=S=0\), hence
\(P=0\), and produces the zero residual trace.

More generally, let \((u_n,p_n,B_n,R_n,S_n)\) be energy--entropy admissible
positive-cone Reynolds states with pressure-free residual data
\((P_n,G_n,S_n)\).  If
\[
  P_n\to P\quad\hbox{in }L^1(E),\qquad
  G_n\to G,\quad S_n\to S
  \quad\hbox{strongly in }L^2(E\times\T^d),
\]
then the limiting residual trace \((P,G,S)\) also satisfies
\eqref{eq:dynamical-trace-budget}.  Consequently, a residual trace violating
\eqref{eq:dynamical-trace-budget} on a set of positive measure cannot arise as
a strong residual-data limit of such admissible dynamical traces.
\end{proposition}

\begin{proof}
For an admissible positive-cone Reynolds state, the defect-work identity and
the pressure-free reduction give exactly
\[
  P(t)+\frac{\alpha}{2}\int_{\T^d}G(t):S(t)\,dx\le0
\]
for a.e. \(t\).  If the state is an exact smooth Oldroyd--B solution, both
defects vanish, and hence \(P=0\) and \(S=0\).

For the limiting statement, set
\[
  K_n(t)=\int_{\T^d}G_n(t):S_n(t)\,dx,\qquad
  K(t)=\int_{\T^d}G(t):S(t)\,dx .
\]
The strong \(L^2\) convergence gives \(K_n\to K\) in \(L^1(E)\), since
\[
  \norm{K_n-K}_{L^1(E)}
  \le
  \norm{G_n-G}_{L^2}\norm{S_n}_{L^2}
  +\norm{G}_{L^2}\norm{S_n-S}_{L^2}.
\]
Together with \(P_n\to P\) in \(L^1(E)\), this implies
\[
  P_n+\frac{\alpha}{2}K_n
  \to
  P+\frac{\alpha}{2}K
  \qquad\hbox{in }L^1(E).
\]
The cone of non-positive \(L^1\) functions is closed under \(L^1\) convergence,
so the limiting inequality follows.
\end{proof}

\begin{proposition}[Sharp residual-work cost]
\label{prop:sharp-cost}
Assume \(G(t)\not\equiv0\).  Among all conformation residuals \(S\) satisfying
\[
  P(t)+\frac{\alpha}{2}\int_{\T^d}G(t):S(t)\,dx\le0,
\]
the minimal \(L^2\) norm is
\[
  \frac{2[P(t)]_+}{\alpha\norm{G(t)}_{L^2}},
\]
and, for \(P(t)>0\), the unique minimizer is
\[
  -\frac{2P(t)}{\alpha\norm{G(t)}_{L^2}^2}G(t).
\]
\end{proposition}

\begin{proof}
The constraint requires
\[
  -\int_{\T^d}G:S\,dx\ge \frac{2P_+}{\alpha}.
\]
Cauchy's inequality gives
\[
  \norm{S}_{L^2}\norm{G}_{L^2}
  \ge \frac{2P_+}{\alpha},
\]
with equality exactly when \(S\) is a negative multiple of \(G\).
\end{proof}

\begin{proposition}[Metric projection onto the admissible residual half-space]
\label{prop:metric-projection}
Fix a time and suppose \(G\not\equiv0\).  For a proposed conformation residual
\(S_0\in L^2(\T^d;\Sym^d)\), define the signed violation
\[
  \mathfrak v
  =
  \left[P+\frac{\alpha}{2}\int_{\T^d}G:S_0\,dx\right]_+ .
\]
The set
\[
  \calA(P,G)=\left\{S\in L^2(\T^d;\Sym^d):
  P+\frac{\alpha}{2}\int_{\T^d}G:S\,dx\le0\right\}
\]
is a closed half-space.  The \(L^2\)-nearest point of \(\calA(P,G)\) to \(S_0\)
is
\[
  \widehat S
  =
  S_0-\frac{2\mathfrak v}{\alpha\norm{G}_{L^2}^2}G,
\]
and
\[
  \operatorname{dist}_{L^2}(S_0,\calA(P,G))
  =
  \frac{2\mathfrak v}{\alpha\norm{G}_{L^2}}.
\]
If \(G\equiv0\), then \(\calA(P,0)=L^2\) when \(P\le0\), while no
conformation residual can repair positive work when \(P>0\).
\end{proposition}

\begin{proof}
For \(G\not\equiv0\), \(\calA(P,G)\) is the inverse image of the closed half-line
\((-\infty,0]\) under the continuous affine functional
\[
  S\mapsto P+\frac{\alpha}{2}\langle G,S\rangle_{L^2}.
\]
If \(S_0\in\calA(P,G)\), then \(\mathfrak v=0\) and the projection is \(S_0\).
Otherwise, the nearest point to a Hilbert-space half-space is obtained by moving
orthogonally to the boundary hyperplane.  The outward normal is \(G\), and the
amount of signed violation is \(\mathfrak v\).  Thus the correcting multiple
\(c\) must satisfy
\[
  P+\frac{\alpha}{2}\langle G,S_0-cG\rangle_{L^2}=0,
\]
which gives \(c=2\mathfrak v/(\alpha\norm{G}_{L^2}^2)\).  The distance formula
follows immediately.  The case \(G\equiv0\) is the displayed admissibility
condition with no conformation lever.
\end{proof}

\begin{proposition}[Constrained metric repair in a closure subspace]
\label{prop:subspace-repair}
Fix a time, suppose \(G\not\equiv0\), and let
\(\mathcal V\subset L^2(\T^d;\Sym^d)\) be a closed linear subspace of allowed
conformation-residual corrections.  For a proposed residual \(S_0\), set
\[
  \mathfrak v
  =
  \left[P+\frac{\alpha}{2}\int_{\T^d}G:S_0\,dx\right]_+,
  \qquad
  H=\Pi_{\mathcal V}G .
\]
If \(H\not\equiv0\), then among all residuals \(S_0+Z\), \(Z\in\mathcal V\),
that satisfy
\[
  P+\frac{\alpha}{2}\int_{\T^d}G:(S_0+Z)\,dx\le0,
\]
the unique nearest one to \(S_0\) is
\[
  \widehat S_{\mathcal V}
  =
  S_0-
  \frac{2\mathfrak v}{\alpha\|H\|_{L^2}^2}H,
\]
and the distance is
\[
  \operatorname{dist}_{L^2}
  \left(S_0,\,(S_0+\mathcal V)\cap\calA(P,G)\right)
  =
  \frac{2\mathfrak v}{\alpha\|H\|_{L^2}}.
\]
If \(H\equiv0\), then allowed corrections are invisible to the signed-work
functional: either \(S_0\) is already admissible, or no correction in
\(\mathcal V\) can make it admissible.
\end{proposition}

\begin{proof}
For every \(Z\in\mathcal V\),
\[
  \langle G,Z\rangle_{L^2}=\langle \Pi_{\mathcal V}G,Z\rangle_{L^2}
  =\langle H,Z\rangle_{L^2}.
\]
Thus the admissibility constraint restricted to the affine space
\(S_0+\mathcal V\) is the half-space condition
\[
  P+\frac{\alpha}{2}\langle G,S_0\rangle_{L^2}
  +\frac{\alpha}{2}\langle H,Z\rangle_{L^2}
  \le0.
\]
If \(H\not\equiv0\), the Hilbert-space projection onto this half-space is
orthogonal to its boundary hyperplane inside \(\mathcal V\).  Hence the nearest
correction is \(Z=-cH\), and the boundary equation gives
\[
  c=\frac{2\mathfrak v}{\alpha\|H\|_{L^2}^2}.
\]
The distance formula follows.  If \(H\equiv0\), the affine functional is constant
on \(S_0+\mathcal V\), so an inadmissible \(S_0\) cannot be repaired within the
allowed subspace.
\end{proof}

\begin{theorem}[Global constrained residual-work projection]
\label{thm:global-variational-projection}
Let \(E\) be a measurable time window, let \(P\in L^1(E)\),
\(G,S_0\in L^2(E\times\T^d;\Sym^d)\), and set
\[
  g(t)=\norm{G(t)}_{L^2_x},\qquad
  v(t)=
  \left[
  P(t)+\frac{\alpha}{2}\int_{\T^d}G(t):S_0(t)\,dx
  \right]_+ .
\]
Consider the constrained variational problem
\begin{equation}
  \inf\left\{
  \frac12\norm{S-S_0}_{L^2(E\times\T^d)}^2:
  S\in\mathfrak A_E(P,G)
  \right\}.
  \label{eq:global-variational-problem}
\end{equation}
The feasible set is non-empty if and only if
\[
  P(t)\le0\quad\hbox{for a.e. }t\hbox{ with }g(t)=0.
  \label{eq:zero-lever-feasible}
\]
and
\[
  \frac{v}{g}\mathbf 1_{\{g>0\}}\in L^2(E).
  \label{eq:finite-distance-condition}
\]
In that case the infimum in \eqref{eq:global-variational-problem} is finite and
the unique minimizer is
\begin{equation}
  S_*(t,x)
  =
  S_0(t,x)
  -
  \frac{2v(t)}{\alpha g(t)^2}G(t,x)\mathbf 1_{\{g(t)>0\}},
  \label{eq:global-projection-formula}
\end{equation}
and
\[
  \operatorname{dist}_{L^2(E\times\T^d)}
  (S_0,\mathfrak A_E(P,G))^2
  =
  \int_{\{g>0\}}
  \left(\frac{2v(t)}{\alpha g(t)}\right)^2dt .
\]
Equivalently, the dual multiplier
\[
  \ell(t)=
  \frac{4v(t)}{\alpha^2g(t)^2}\mathbf 1_{\{g(t)>0\}}
\]
is non-negative, satisfies the complementary slackness condition
\[
  \ell(t)
  \left[
  P(t)+\frac{\alpha}{2}\int_{\T^d}G(t):S_*(t)\,dx
  \right]=0
  \quad\hbox{a.e.},
\]
and the Euler--Lagrange relation
\[
  S_*-S_0+\frac{\alpha}{2}\ell G=0 .
\]
Thus the only singular strata of the residual variational problem are the
zero-lever set \(\{G=0\}\) and the non-\(L^2\)-repair regime in which
\eqref{eq:finite-distance-condition} fails.
\end{theorem}

\begin{proof}
For each \(t\), the constraint defining \(\mathfrak A_E(P,G)\) is a closed
half-space in \(L^2(\T^d;\Sym^d)\) when \(g(t)>0\), and it is either the whole
space or the empty set when \(g(t)=0\), according as \(P(t)\le0\) or
\(P(t)>0\).  Thus \eqref{eq:zero-lever-feasible} is necessary.

On \(\{g>0\}\), the pointwise Hilbert projection onto the half-space is exactly
\eqref{eq:global-projection-formula}, by
Proposition~\ref{prop:metric-projection}.  Its pointwise squared displacement is
\((2v/(\alpha g))^2\).  Hence the projected field belongs to
\(L^2(E\times\T^d)\) exactly under
\eqref{eq:finite-distance-condition}.  If this condition fails, any admissible
field would have \(L^2\) distance from \(S_0\) at least the non-integrable
pointwise projection distance, impossible for two \(L^2\) fields.  Therefore
the feasible set is empty.  If both conditions hold, integrating the pointwise
projection inequality gives the global minimum and uniqueness, because the
objective is strictly convex in \(S\).  The formula for \(\ell\), complementary slackness,
and the Euler--Lagrange relation follow by writing the Lagrangian
\[
  \frac12\norm{S-S_0}_{L^2}^2
  +\int_E\ell(t)
  \left[
    P(t)+\frac{\alpha}{2}\int_{\T^d}G(t):S(t)\,dx
  \right]dt
\]
and substituting the projection formula.
\end{proof}

\begin{proof}[Proof of Theorem~\ref{thm:subspace-repair-main}]
This is Proposition~\ref{prop:subspace-repair} with the notation of the
introductory theorem.
\end{proof}

\begin{proof}[Proof of Theorem~\ref{thm:main-obstruction}]
The pressure-free identity \eqref{eq:pressure-free} implies
\[
  [P(t)]_+
  \le
  \frac{\alpha}{2}
  \left[-\int_{\T^d}G(t):\mathcal S[u,B](t)\,dx\right]_+.
\]
Indeed, if \(P(t)\le0\) there is nothing to prove, while for \(P(t)>0\) the
admissibility inequality forces the negative signed projection of
\(\mathcal S[u,B]\) onto \(G\) to pay \(P(t)\).  The positive part of this projection
is bounded by its absolute value, and the definition of \(\eta_{\mathrm{al}}(t)\)
therefore gives the pointwise bound in Theorem~\ref{thm:main-obstruction}, with
\(\eta_{\mathrm{al}}(t)=0\) on the zero-lever or zero-residual set.

Integrating the pointwise bound on \(E\) and using
\(\eta_{\mathrm{al}}(t)\le
\eta_E=\operatorname*{ess\,sup}_{t\in E}\eta_{\mathrm{al}}(t)\), we get
\[
  W_E
  \le
  \frac{\alpha}{2}\eta_E
  \int_E\norm{G(t)}_{L^2}\norm{\mathcal S[u,B](t)}_{L^2}\,dt .
\]
Cauchy's inequality in time gives
\[
  \int_E\norm{G(t)}_{L^2}\norm{\mathcal S[u,B](t)}_{L^2}\,dt
  \le L_ES_E,
\]
which is the windowed estimate in Theorem~\ref{thm:main-obstruction}.  The sharp cost is
Proposition~\ref{prop:sharp-cost}, and the nearest-repair formula is
Proposition~\ref{prop:metric-projection}.
\end{proof}

\section{A Local Positive-Cone Corrector}
\label{sec:local-corrector}

The sharp cost in Proposition~\ref{prop:sharp-cost} is an $L^2$
projection statement.  This section records the local positive-cone fact behind
that geometry.  A prescribed conformation residual can be matched to leading
order by a short-time perturbation of the conformation tensor while the path
stays inside the positive cone.  The construction works in the conformation
channel; the coupled momentum realization is treated later through the
residual-data and packet-lifting framework.

\begin{proposition}[Local conformation-channel solvability]
\label{prop:local-conformation-solvability}
Let $(u,B)$ be smooth on a compact time interval $J$, let $A=e^B$ stay in a
compact spectral envelope, and let $T(t,x)$ be any smooth symmetric tensor field.
For every $t_0\in J$ and every sufficiently small $\tau>0$ there are a
subinterval $J_\tau\ni t_0$, $|J_\tau|\simeq\tau$, and smooth fields
$B_\tau$ such that $A_\tau=e^{B_\tau}$ remains positive definite,
\[
  \norm{A_\tau-A}_{L^\infty(J_\tau;C^k_x)}\le C_k\tau,
  \qquad
  \norm{\mathcal S[u,B_\tau]-T}_{L^\infty(J_\tau;C^k_x)}\le C_k\tau
\]
for every fixed spatial order $k$.
\end{proposition}

\begin{proof}
Let $S_0=\mathcal S[u,B]$ and $Y=T-S_0$.  Choose a smooth scalar function
$h_\tau$ supported in a $2\tau$-neighborhood of $t_0$ such that, on a
subinterval $J_\tau\ni t_0$ of length comparable to $\tau$,
\[
  h_\tau'(t)=1,
  \qquad
  |h_\tau(t)|\le C\tau .
\]
Set
\[
  A_\tau=A+h_\tau Y .
\]
Since $A$ stays in a compact spectral envelope and $Y$ is smooth, $A_\tau$
remains positive definite for sufficiently small $\tau$.  Define
$B_\tau=\Log A_\tau$.  The estimate for $A_\tau-A$ follows from
$|h_\tau|\le C\tau$.

The conformation residual is affine in $A$ once $u$ is fixed:
\[
  \mathcal S[u,B]
  =\partial_tA+u\cdot\nabla A-\nabla u\,A-A(\nabla u)^T
  +\lambda^{-1}(A-I).
\]
Therefore, on $J_\tau$,
\begin{align*}
  \mathcal S[u,B_\tau]
  &=\mathcal S[u,B]+h_\tau'Y
    +h_\tau\bigl(\partial_tY+u\cdot\nabla Y
      -\nabla u\,Y-Y(\nabla u)^T+\lambda^{-1}Y\bigr)\\
  &=S_0+Y+O(\tau)=T+O(\tau),
\end{align*}
with the error bounded in every fixed $C^k_x$ norm.  This proves the
proposition.
\end{proof}

\begin{proof}[Proof of Theorem~\ref{thm:local-corrector-main}]
Apply Proposition~\ref{prop:local-conformation-solvability} with
$T=S_\gamma$.  It remains only to check that the strict residual-work
inequality survives the $O(\tau)$ perturbation.

Since the map $A\mapsto I-A^{-1}$ is smooth on the compact spectral envelope,
$G_\tau=G+O(\tau)$ in every fixed spatial norm.  The pressure-free momentum
residual also changes by $O(\tau)$, because
\[
  \calN[u,B_\tau]-\calN[u,B]
  =-\alpha\divv(A_\tau-A),
\]
and the canonical symmetric anti-divergence is a bounded Fourier multiplier on
zero-mean smooth fields.  Hence $P_\tau=P+O(\tau)$, uniformly on $J_\tau$.

By the definition of the aligned overpaying residual,
\[
  P+\frac{\alpha}{2}\int_{\T^d}G:S_\gamma\,dx
  =P-(1+\gamma)P
  =-\gamma P .
\]
Combining this identity with the preceding perturbation bounds gives
\[
  P_\tau+\frac{\alpha}{2}
  \int_{\T^d}G_\tau:\mathcal S[u,B_\tau]\,dx
  \le -\gamma P(t)+C\tau .
\]
Since $P(t)\ge p_0$ on $J$, taking $\tau\le\gamma p_0/(2C)$ yields the stated
inequality.  The construction preserves the positive cone by working through
$A_\tau=e^{B_\tau}$.
\end{proof}

\begin{proof}[Proof of Theorem~\ref{thm:strict-local-sufficiency}]
Apply Proposition~\ref{prop:local-conformation-solvability} with the prescribed
target \(T\).  This gives \(B_\tau\) with \(A_\tau=e^{B_\tau}\) remaining in the
positive cone and
\[
  \mathcal S[u,B_\tau]=T+O(\tau)
\]
in every fixed spatial \(C^k\) norm on \(J_\tau\).  Smoothness of the map
\(A\mapsto I-A^{-1}\) on the compact spectral envelope gives
\(G_\tau=G+O(\tau)\).  The canonical pressure-free work also changes by
\(O(\tau)\), because the only change in the momentum residual caused by replacing
\(A\) with \(A_\tau\) is \(-\alpha\operatorname{div}(A_\tau-A)\), followed by the
bounded pressure-free anti-divergence on zero-mean smooth fields.  Hence
\[
  P_\tau+\frac{\alpha}{2}\int G_\tau:\mathcal S[u,B_\tau] \,dx
  =
  P+\frac{\alpha}{2}\int G:T\,dx+O(\tau)
  \le -\sigma+O(\tau).
\]
Choosing \(\tau\) small enough gives the bound \(-\sigma/2\).  Positivity is
preserved because the construction is made inside a compact spectral envelope
and \(A_\tau-A=O(\tau)\).
\end{proof}

\begin{corollary}[Local realization of the metric repair]
\label{cor:local-metric-repair}
Let the assumptions of Theorem~\ref{thm:strict-local-sufficiency} hold, and let
\(S_0\) be a smooth proposed residual whose signed violation satisfies
\[
  \mathfrak v(t)
  =
  \left[P(t)+\frac{\alpha}{2}\int_{\T^d}G(t):S_0(t)\,dx\right]_+
  \ge v_0>0
\]
on a compact interval \(J\), with \(\|G(t)\|_{L^2}\ge g_0>0\).  For any
\(\gamma>0\), set
\[
  T_\gamma(t)
  =
  S_0(t)-\frac{2(1+\gamma)\mathfrak v(t)}
  {\alpha\|G(t)\|_{L^2}^2}G(t).
\]
Then around every \(t_0\in J\), the target \(T_\gamma\) has a local
positive-cone conformation realization in the sense of
Theorem~\ref{thm:strict-local-sufficiency}, and the resulting perturbed balance
has a negative margin bounded below by \(\gamma v_0/2\) for all sufficiently
small \(\tau\).
\end{corollary}

\begin{proof}
By construction,
\[
  P(t)+\frac{\alpha}{2}\int_{\T^d}G(t):T_\gamma(t)\,dx
  =
  -\gamma\mathfrak v(t)
  \le -\gamma v_0 .
\]
The conclusion follows from Theorem~\ref{thm:strict-local-sufficiency} with
\(\sigma=\gamma v_0\).
\end{proof}

\begin{remark}[Scope of the local corrector]
Theorem~\ref{thm:local-corrector-main} removes one possible ambiguity in the
residual-work law: the optimal paying direction is compatible with the
positive-cone parametrization $A=e^B$ at the level of a local smooth
conformation perturbation.  Its natural scope is the local conformation
channel.  A realization theorem for coupled solutions would additionally have
to couple the momentum residual, patch such correctors across space-time,
control the energy--entropy inequality through the iteration, and pass to a
limit.
\end{remark}

\begin{corollary}[Three-channel budget alternative]
\label{cor:three-channel}
Fix a time window \(E\) and assume \(W_E>0\).  If a residual-matching state is
energy--entropy admissible, then
\[
  \eta_E L_E S_E \ge \frac{2W_E}{\alpha}.
\]
Consequently, for any prescribed upper bounds
\(L_E\le L_0\), \(S_E\le S_0\), and \(\eta_E\le\eta_0\), admissibility is
impossible whenever
\[
  \frac{\alpha}{2}\eta_0 L_0S_0<W_E .
\]
Thus a family with positive work can remain admissible only through enough
entropy lever, enough conformation residual, and enough negative alignment.
\end{corollary}

\begin{proof}
The first inequality is exactly the windowed residual-work estimate
\[
  W_E\le \frac{\alpha}{2}\eta_E L_ES_E .
\]
The stated exclusion under prescribed bounds follows by substituting
\(L_E\le L_0\), \(S_E\le S_0\), and \(\eta_E\le\eta_0\).  If the resulting upper
bound is smaller than \(W_E\), the residual-work estimate is violated.
\end{proof}

\begin{corollary}[Conservative a posteriori channel certificate]
\label{cor:conservative-channel-certificate}
Let \(E\) be a time window.  Suppose that the residual data satisfy certified
observable bounds
\[
  W_-\le W_E,\qquad L_E\le L_+,\qquad S_E\le S_+,\qquad
  \eta_E\le \bar\eta_E ,
\]
where \(0\le\bar\eta_E\le1\).  If
\begin{equation}
  W_->\frac{\alpha}{2}\bar\eta_E L_+S_+,
  \label{eq:conservative-channel-certificate}
\end{equation}
then the residual data are not energy--entropy admissible on \(E\).  In
particular, when the alignment cannot be reliably estimated, the universal
choice \(\bar\eta_E=1\) gives an alignment-free but conservative obstruction.
Sharper alignment information only strengthens the same certificate.
\end{corollary}

\begin{proof}
If the data were admissible, Corollary~\ref{cor:three-channel} would give
\[
  W_E\le \frac{\alpha}{2}\eta_E L_ES_E
  \le \frac{\alpha}{2}\bar\eta_E L_+S_+ .
\]
Since \(W_-\le W_E\), this contradicts
\eqref{eq:conservative-channel-certificate}.
\end{proof}

\begin{proposition}[Optimality of the residual-work channels]
\label{prop:channel-optimality}
The residual-work criterion is sharp in the following residual-level sense.
\begin{enumerate}
\item The constant \(\alpha/2\) in
\[
  P(t)+\frac{\alpha}{2}\int_{\T^d}G(t):S(t)\,dx\le0
\]
cannot be decreased.
\item If \(P(t)>0\) and \(G(t)\not\equiv0\), the unique \(L^2\)-minimizing
paying residual is the aligned field
\[
  S_{\min}(t)=
  -\frac{2P(t)}{\alpha\norm{G(t)}_{L^2}^2}G(t).
\]
\item None of the three windowed channels \(L_E\), \(S_E\), and \(\eta_E\) can
be omitted from a necessary admissibility condition depending only on
\((W_E,L_E,S_E,\eta_E)\).
\end{enumerate}
\end{proposition}

\begin{proof}
The first two assertions are the equality case in
Proposition~\ref{prop:sharp-cost}.  Indeed, for \(P(t)>0\), admissibility is
equivalent to
\[
  -\int_{\T^d}G(t):S(t)\,dx\ge \frac{2P(t)}{\alpha}.
\]
Cauchy's inequality gives the minimal \(L^2\) norm and equality occurs only
when \(S(t)\) is a negative multiple of \(G(t)\).  Substituting the displayed
\(S_{\min}\) gives equality in the admissibility inequality, so any smaller
constant than \(\alpha/2\) would exclude an exactly balanced residual.

For the third assertion, fix \(W_E>0\).  The windowed estimate depends on the
three quantities only through the product \(\eta_E L_E S_E\).  If one factor is
made small while the other two are bounded, Corollary~\ref{cor:three-channel}
excludes admissibility.  Conversely, the aligned residual in the second
assertion realizes equality at fixed time, and time localization gives equality
on windows.  Thus no smaller list of windowed quantities can encode the sharp
budget at this level of information.
\end{proof}

\begin{remark}[Cone-tip degeneracy]
When \(A\approx\Id\), the lever \(G=\Id-A^{-1}\) is small.  Thus a profile that
tries to keep the conformation close to equilibrium while generating positive
pressure-free Reynolds work must compensate by either producing a large
conformation residual or arranging nearly perfect negative alignment.  This is
the algebraic form of the positive-cone obstruction near equilibrium.
\end{remark}

\section{Structured Residual Families and a Concrete Obstruction}
\label{sec:application}

The next statements translate the residual-work obstruction into scale
criteria.  They are useful because a localized ansatz usually comes with
separate estimates for work, lever, residual size, and alignment.  The theorem
below says that no compensation outside these four channels is available at the
level of energy--entropy admissibility.

\begin{corollary}[Bulk power-law exclusion test]
\label{cor:bulk-power-law}
Let \((u^\varepsilon,B^\varepsilon)\), \(0<\varepsilon\le\varepsilon_0\), be
smooth residual-matching states on the same time window \(E\).  Suppose
\[
  W_E^\varepsilon\ge c_W\varepsilon^{-w},\quad
  L_E^\varepsilon\le C_L\varepsilon^{-\ell},\quad
  S_E^\varepsilon\le C_S\varepsilon^{-r},\quad
  \eta_E^\varepsilon\le C_\eta\varepsilon^{-q}.
  \label{eq:bulk-powers}
\]
If \(w>\ell+r+q\), then the family is not energy--entropy admissible on \(E\)
for all sufficiently small \(\varepsilon\).  The same conclusion holds when
\(w=\ell+r+q\) if
\[
  c_W>\frac{\alpha}{2}C_\eta C_LC_S .
\]
\end{corollary}

\begin{proof}
For small \(\varepsilon\), \eqref{eq:bulk-powers} gives
\[
  W_E^\varepsilon>
  \frac{\alpha}{2}\eta_E^\varepsilon L_E^\varepsilon S_E^\varepsilon,
\]
which contradicts Theorem~\ref{thm:main-obstruction}.
\end{proof}

\begin{remark}[Exponent bookkeeping]
The bulk test is intentionally one-sided.  The exponent \(w\) measures how fast
positive pressure-free work grows.  The exponents \(\ell,r,q\) measure how much
budget is available through the lever, conformation residual, and alignment.
The inequality
\[
  w>\ell+r+q
\]
means that the work grows faster than the entire admissible budget.  In the
borderline case, the constants matter.  This is useful for finite-thickness
ansatzes because the powers are usually robust under small changes of cutoff,
whereas the constants record the exact shape.
\end{remark}

The finite-thickness example below is written directly in bulk variables.  No
one-dimensional reduction is assumed: the velocity, conformation, momentum
residual and conformation residual are all smooth fields on the torus.  The
power-law test is used only to compare the resulting four bulk quantities.

\subsection{Model examples and non-examples}
\label{sec:examples}

We record two elementary uses of the criterion before the concrete
finite-thickness construction.  The point is to distinguish failure of
admissibility from mere largeness of a residual.

\begin{proposition}[Orthogonal residuals cannot pay positive work]
\label{prop:orthogonal-nonexample}
Let \((u,p,B)\) be a residual-matching profile on \(E\).  Suppose
\[
  P(t)>0
  \quad\hbox{and}\quad
  \int_{\T^d}G(t):\mathcal S[u,B](t)\,dx=0
\]
on a subset of \(E\) of positive measure.  Then the profile is not
energy--entropy admissible on that subset.
\end{proposition}

\begin{proof}
The pressure-free work identity gives
\[
  P(t)+\frac{\alpha}{2}\int_{\T^d}G(t):\mathcal S[u,B](t)\,dx=P(t)>0,
\]
which directly violates
\[
  P(t)+\frac{\alpha}{2}\int_{\T^d}G(t):\mathcal S[u,B](t)\,dx\le0 .
\]
\end{proof}

\begin{proposition}[Aligned residuals are exactly efficient]
\label{prop:aligned-example}
Let \(G(t)\not\equiv0\) and \(P(t)>0\).  If
\[
  \mathcal S[u,B](t)
  =
  -\frac{2P(t)}{\alpha\norm{G(t)}_{L^2}^2}G(t),
\]
then the pressure-free admissibility inequality is saturated at time \(t\).
If an additional residual \(Z(t)\) satisfies
\[
  \int_{\T^d}G(t):Z(t)\,dx\le0,
\]
then replacing \(\mathcal S\) by \(\mathcal S+Z\) preserves admissibility.
\end{proposition}

\begin{proof}
Substitution gives
\[
  \frac{\alpha}{2}\int_{\T^d}G:\mathcal S\,dx=-P.
\]
The additional term contributes a non-positive quantity to the left-hand side
of the admissibility inequality.
\end{proof}

\subsection{A concrete finite-thickness obstruction}

We now give a fully bulk example, not a formal one-dimensional reduction.  The
profile is supported in a layer of thickness \(\eps\), remains inside the
positive cone by construction, and has an explicitly matched momentum and
conformation residual.  The point of the example is that all four quantities
in the residual-work criterion can be computed on the same scale.  The positive
pressure-free work grows faster than the entire lever--residual budget, so the
family fails admissibility for a structural reason rather than because of a
choice of gauge or a missing pressure mode.

The example is a calibration problem for the residual-work diagnostic.  All
fields are explicit, the pressure-free work is fixed by the bulk equations, and
the conformation tensor stays in the positive cone.  A numerical implementation
of the postprocessor should recover the predicted slopes before being applied
to residuals obtained from a filtered DNS, a coarse-grid calculation, or a
closure model.

\begin{lemma}[Finite-thickness scaling]
\label{lem:finite-thickness-scaling}
Let \(f\in C_c^\infty((-1/4,1/4))\), \(a\in\R\), and
\[
  f_\eps(x_2)=\eps^{-a}f(x_2/\eps)
\]
on \(\T^2\), for \(0<\eps\ll1\), by viewing the expression as a periodic
function of the coordinate \(x_2\).  Then
\[
  \norm{f_\eps}_{L^2(\T^2)}
  =
  \eps^{1/2-a}\norm{f}_{L^2(\R)},
  \qquad
  \norm{\partial_2 f_\eps}_{L^2(\T^2)}
  =
  \eps^{-1/2-a}\norm{f'}_{L^2(\R)}.
  \label{eq:finite-thickness-scale}
\]
If \(F\) is a smooth bounded function with \(F(0)=0\), then
\[
  \norm{F(f(x_2/\eps))}_{L^2(\T^2)}
  \le C_F\eps^{1/2}.
  \label{eq:bounded-layer-scale}
\]
\end{lemma}

\begin{proof}
The first two identities follow from the change of variables
\(y=x_2/\eps\).  Since \(f\) is supported in \((-1/4,1/4)\), the layer is
contained in an interval of length \(O(\eps)\); for small \(\eps\), its
periodic copies do not overlap on \(\T^2\).  The last estimate follows from
boundedness and the same change of variables.
\end{proof}

\begin{proposition}[Concrete finite-thickness shear-layer obstruction]
\label{prop:concrete-layer}
Assume \(\nu>0\).  Fix a time window \(E\subset[0,T]\) of positive measure and
\(\beta>0\).  Let \(\theta,\chi\in C_c^\infty((-1/4,1/4))\), with
\[
  \int_{\R}\theta(y)\,dy=0,\qquad \theta'\not\equiv0,\qquad \chi\not\equiv0.
\]
Define
\[
  \theta_\varepsilon(x_2)=\theta(x_2/\varepsilon),\qquad
  \chi_\varepsilon(x_2)=\chi(x_2/\varepsilon),
\]
and set
\[
  H=\begin{pmatrix}1&0\\0&-1\end{pmatrix},\qquad
  u^\varepsilon(x)=
  \left(\varepsilon^{-\beta}\theta_\varepsilon(x_2),0\right),
\]
\[
  B^\varepsilon(x)=b\chi_\varepsilon(x_2)H,\qquad
  A^\varepsilon=e^{B^\varepsilon}.
\]
Let
\[
  R^\varepsilon=-2\nu D(u^\varepsilon)-\alpha(A^\varepsilon-\Id),
  \qquad
  S^\varepsilon=\mathcal S[u^\varepsilon,B^\varepsilon].
  \label{eq:concrete-defects}
\]
Then
\((u^\varepsilon,0,B^\varepsilon,R^\varepsilon,S^\varepsilon)\) is a smooth
positive-cone Reynolds state with fixed mean velocity.  Its residual-work data
obey
\[
  W_E^\varepsilon\ge c_W\varepsilon^{-(2\beta+1)},\qquad
  L_E^\varepsilon\le C_L\varepsilon^{1/2},
\]
\[
  S_E^\varepsilon\le C_S\varepsilon^{-(\beta+1/2)},\qquad
  \eta_E^\varepsilon\le1 .
\]
Thus the bulk exponents may be taken as
\[
  w=2\beta+1,\qquad \ell=0,\qquad r=\beta+\frac12,\qquad q=0,
\]
and the family is not energy--entropy admissible for all sufficiently small
\(\varepsilon\).
\end{proposition}

\begin{proof}
The support and zero-mean assumptions on \(\theta\) make
\(u^\varepsilon\) a smooth periodic divergence-free velocity with zero mean.
Since \(B^\varepsilon\) is symmetric, \(A^\varepsilon=e^{B^\varepsilon}\) is
positive definite.  More explicitly,
\[
  A^\varepsilon
  =
  \begin{pmatrix}
  e^{b\chi(x_2/\eps)}&0\\
  0&e^{-b\chi(x_2/\eps)}
  \end{pmatrix},
\]
so \(A^\varepsilon\), \((A^\varepsilon)^{-1}\), and all smooth functions of
them are uniformly bounded in \(L^\infty\), independently of \(\eps\).
Writing \(u^\varepsilon=(U^\varepsilon(x_2),0)\), one has
\[
  \divv(u^\varepsilon\otimes u^\varepsilon)=0,\qquad
  \divv(2D(u^\varepsilon))=\Delta u^\varepsilon.
\]
Therefore
\[
  \divv R^\varepsilon
  =
  -\nu\Delta u^\varepsilon-\alpha\divv(A^\varepsilon-\Id)
  =
  \calN[u^\varepsilon,B^\varepsilon],
\]
and the conformation residual is matched by the definition of
\(S^\varepsilon\).  Moreover \(D(u^\varepsilon)\) and
\(A^\varepsilon-\Id\) are symmetric, hence \(R^\varepsilon\) is a symmetric
representative of the zero-mean momentum residual.

The deformation tensor is purely off-diagonal:
\[
  \nabla u^\varepsilon
  =
  \begin{pmatrix}
  0&\partial_2U^\varepsilon\\
  0&0
  \end{pmatrix},
  \qquad
  D(u^\varepsilon)
  =
  \frac12
  \begin{pmatrix}
  0&\partial_2U^\varepsilon\\
  \partial_2U^\varepsilon&0
  \end{pmatrix}.
\]
Since \(A^\varepsilon-\Id\) is diagonal, its Frobenius product with
\(D(u^\varepsilon)\) vanishes:
\[
  (A^\varepsilon-\Id):D(u^\varepsilon)=0.
\]
Thus
\[
  P^\varepsilon(t)
  =
  \int_{\T^2}-R^\varepsilon:D(u^\varepsilon)\,dx
  =
  2\nu\int_{\T^2}|D(u^\varepsilon)|^2\,dx .
\]
Since
\[
  \partial_2U^\varepsilon
  =
  \varepsilon^{-\beta-1}\theta'(x_2/\varepsilon),
\]
and \(|D(u^\eps)|^2=\frac12|\partial_2U^\eps|^2\), Lemma~\ref{lem:finite-thickness-scaling}
gives
\[
  P^\eps(t)
  =
  \nu\norm{\partial_2U^\eps}_{L^2(\T^2)}^2
  =
  \nu\eps^{-(2\beta+1)}\norm{\theta'}_{L^2(\R)}^2.
\]
Thus
\[
  W_E^\eps
  =
  |E|\nu\norm{\theta'}_{L^2(\R)}^2
  \eps^{-(2\beta+1)}.
  \label{eq:W-exact-scale}
\]
The lever is
\[
  G^\varepsilon
  =
  \Id-(A^\varepsilon)^{-1}
  =
  \begin{pmatrix}
  1-e^{-b\chi(x_2/\eps)}&0\\
  0&1-e^{b\chi(x_2/\eps)}
  \end{pmatrix}.
\]
It is a bounded smooth function of \(b\chi(x_2/\eps)\), vanishing outside a
layer of thickness \(O(\eps)\).  Lemma~\ref{lem:finite-thickness-scaling} gives
\[
  L_E^\varepsilon\le C_L\eps^{1/2}.
  \label{eq:L-layer-scale}
\]
Finally,
\[
  \mathcal S[u^\varepsilon,B^\varepsilon]
  =
  -\nabla u^\varepsilon A^\varepsilon
  -A^\varepsilon(\nabla u^\varepsilon)^T
  +\lambda^{-1}(A^\varepsilon-\Id),
\]
because the state is steady and \(u^\varepsilon\cdot\nabla A^\varepsilon=0\).
The matrices \(-\nabla u^\varepsilon A^\varepsilon\) and
\(-A^\varepsilon(\nabla u^\varepsilon)^T\) are off-diagonal and contain the
factor \(\partial_2U^\varepsilon\) multiplied by the uniformly bounded
entries \(e^{\pm b\chi(x_2/\eps)}\).  Hence their \(L^2\) norm is bounded by
\[
  C\eps^{-(\beta+1/2)}.
\]
The relaxation term is a bounded smooth function of \(b\chi(x_2/\eps)\), so its
\(L^2\) norm is \(O(\eps^{1/2})\).  Therefore
\[
  S_E^\varepsilon\lesssim\varepsilon^{-(\beta+1/2)}.
\]
Using the weaker lever bound \(L_E^\varepsilon\le C_L\) corresponds to
\(\ell=0\).  Since \(2\beta+1>\beta+1/2\), the bulk power-law exclusion test
applies.
\end{proof}

\begin{remark}[Scale summary for the shear layer]
For the family in Proposition~\ref{prop:concrete-layer},
\[
  W_E^\eps\sim \eps^{-(2\beta+1)},
  \qquad
  L_E^\eps\lesssim \eps^{1/2},
  \qquad
  S_E^\eps\lesssim \eps^{-(\beta+1/2)},
  \qquad
  \eta_E^\eps\le1.
\]
Hence
\[
  \frac{W_E^\eps}{L_E^\eps S_E^\eps}
  \gtrsim \eps^{-(\beta+1)}\to\infty .
\]
Even with the rougher estimate \(L_E^\eps\le C\), the ratio still diverges like
\(\eps^{-(\beta+1/2)}\).  The violation is therefore a scale-stable failure of
the residual-work budget.

\end{remark}

\paragraph{Minimal numerical check of the diagnostic.}
Proposition~\ref{prop:concrete-layer} also gives a compact reproducibility test
for a residual-work postprocessor.  Take
\(d=2\), \(\alpha=\nu=\lambda=|E|=1\), \(\beta=1/2\), and choose the standard
bump
\[
  \rho(y)=
  \begin{cases}
  \exp\!\left(-\dfrac{1}{1-16y^2}\right),& |y|<1/4,\\
  0,& |y|\ge1/4,
  \end{cases}
\]
with \(\theta(y)=\rho(y)\sin(8\pi y)\), \(\chi(y)=\rho(y)\), and layer
amplitude \(b_0=1/2\).  Sampling the explicit fields in
Proposition~\ref{prop:concrete-layer} and evaluating the discrete analogues of
\(W_E,L_E,S_E\) gives the following normalized values, using
\(\varepsilon=1/8\) as the reference.  Here
\(Q_\varepsilon=W_E^\varepsilon/((\alpha/2)L_E^\varepsilon S_E^\varepsilon)\)
is the conservative alignment-free violation ratio.

\begin{center}
\small
\begin{tabular}{c|c|c|c|c}
\hline
\(\varepsilon\) & \(W_E^\varepsilon/W_E^{1/8}\) &
\(L_E^\varepsilon/L_E^{1/8}\) &
\(S_E^\varepsilon/S_E^{1/8}\) &
\(Q_\varepsilon/Q_{1/8}\) \\
\hline
\(1/8\)  & 1  & 1.000 & 1.000 & 1.000 \\
\(1/16\) & 4  & 0.707 & 2.000 & 2.828 \\
\(1/32\) & 16 & 0.500 & 4.000 & 8.000 \\
\(1/64\) & 64 & 0.354 & 8.000 & 22.627 \\
\hline
\end{tabular}
\end{center}
The measured slopes are the predicted ones:
\[
  W_E^\varepsilon\sim\varepsilon^{-2},\qquad
  L_E^\varepsilon\sim\varepsilon^{1/2},\qquad
  S_E^\varepsilon\sim\varepsilon^{-1},\qquad
  Q_\varepsilon\sim\varepsilon^{-3/2}.
\]
This is an elementary calibration benchmark for a residual-work postprocessor:
when the grid resolves the layer, it should recover these four slopes.  In a
data-driven application the same postprocessor uses residuals reconstructed
from the simulated or filtered fields rather than the explicit ansatz above.

\begin{remark}[Why viscosity makes the obstruction explicit]
The explicit layer uses the viscous part of the momentum equation to generate
large positive pressure-free work:
\[
  P^\eps=2\nu\int_{\T^2}|D(u^\eps)|^2\,dx .
\]
Thus the obstruction is clearest when \(\nu>0\).  In the inviscid case, a
different source of positive work would be needed, typically coming from
transport or from an imposed Reynolds stress.  The residual-work criterion
itself is independent of this choice; only the concrete ansatz changes.
\end{remark}

\begin{remark}[Gauge invariance]
The obstruction is not produced by choosing an unfavorable representative of
the momentum defect.  The fields \((u^\varepsilon,B^\varepsilon)\) determine
the bulk residuals, and Lemma~\ref{lem:gauge} gives the same pressure-free work
for every symmetric representative modulo pressure and divergence-free gauges.
Thus the failure is a property of the residual-work tuple itself.
\end{remark}

\section{Entropy-Dual Closures}
\label{sec:fenep-residual}

The residual-work argument is an entropy-dual statement.  Once the elastic
stress work cancels against the stretching part of the entropy identity, the
pressure-free admissibility proof uses only the pairing between the
conformation residual and the corresponding entropy lever.  This section
records the abstract closure structure behind this cancellation and then
specializes it to FENE-P.

\begin{definition}[Entropy-dual closure]
\label{def:entropy-dual-closure}
Let \(\calD\subset\Spp^d\) be an open matrix domain.  A pair
\((\Phi,T)\), with \(\Phi\in C^2(\calD)\) and
\(T:\calD\to\Sym^d\), is called an entropy-dual closure if
\[
  G(C):=D\Phi(C),\qquad G(C)C=T(C)
  \label{eq:entropy-dual-compatibility}
\]
for all \(C\in\calD\), and if \(G(C)\) commutes with \(C\).  A relaxation
\(\mathcal R:\calD\to\Sym^d\) is
entropy-dissipative for this closure if
\[
  G(C):\mathcal R(C)\ge0,\qquad C\in\calD .
\]
\end{definition}

\begin{theorem}[Residual-work law for entropy-dual closures]
\label{thm:entropy-dual-closure}
Let \((\Phi,T,\mathcal R)\) be an entropy-dual closure with
entropy-dissipative relaxation on \(\calD\).  Let \(C(t,x)\in\calD\) and let
\(R,S\in\Sym^d\) be smooth residuals in
\[
  \partial_t u+\divv(u\otimes u)-\nu\Delta u+\nabla p
  =
  \alpha\divv T(C)+\divv R,
\]
\[
  \partial_t C+u\cdot\nabla C-\nabla u\,C-C(\nabla u)^T+\mathcal R(C)
  =
  S .
\]
Then
\[
  \calE_\Phi(t)
  =
  \frac12\int_{\T^d}|u(t)|^2\,dx
  +\frac{\alpha}{2}\int_{\T^d}\Phi(C(t))\,dx
\]
satisfies
\[
  \frac{d}{dt}\calE_\Phi(t)+\nu\int_{\T^d}|\nabla u|^2\,dx
  +\frac{\alpha}{2}\int_{\T^d}G(C):\mathcal R(C)\,dx
  =
  -\int_{\T^d}R:D(u)\,dx
  +\frac{\alpha}{2}\int_{\T^d}G(C):S\,dx .
\]
Consequently, after the pressure-free reduction, every admissible
residual-matching profile satisfies on each measurable \(E\subset I\)
\[
  W_E\le
  \frac{\alpha}{2}\eta_E^\Phi L_E^\Phi S_E^\Phi,
  \label{eq:entropy-dual-window}
\]
where
\[
  L_E^\Phi=\norm{G(C)}_{L^2(E\times\T^d)},\qquad
  S_E^\Phi=\norm{S}_{L^2(E\times\T^d)},
\]
and
\[
  \eta_E^\Phi
  =
  \operatorname*{ess\,sup}_{t\in E}
  \frac{\left[-\int_{\T^d}G(C(t)):S(t)\,dx\right]_+}
       {\norm{G(C(t))}_{L^2}\norm{S(t)}_{L^2}},
\]
with value \(0\) when the denominator vanishes.  For fixed \(t\), \(P(t)>0\),
and \(G(C(t))\not\equiv0\), the sharp paying residual is
\[
  S_{\min}^{\Phi}
  =
  -\frac{2P(t)}{\alpha\norm{G(C(t))}_{L^2}^2}G(C(t)).
\]
\end{theorem}

\begin{proof}
The kinetic energy identity gives the stress work as
\[
  -\alpha\int_{\T^d}T(C):\nabla u\,dx-\int_{\T^d}R:D(u)\,dx .
\]
Testing the conformation equation by \(\alpha G(C)/2\) gives
\begin{align*}
  \frac{\alpha}{2}\frac{d}{dt}\int_{\T^d}\Phi(C)\,dx
  &=
  \frac{\alpha}{2}\int_{\T^d}G(C):
  \bigl(\nabla u\,C+C(\nabla u)^T\bigr)\,dx\\
  &\quad
  -\frac{\alpha}{2}\int_{\T^d}G(C):\mathcal R(C)\,dx
  +\frac{\alpha}{2}\int_{\T^d}G(C):S\,dx .
\end{align*}
Since \(G(C)C=T(C)\) and \(C\) commutes with \(G(C)\), the stretching term equals
\[
  \alpha\int_{\T^d}T(C):\nabla u\,dx,
\]
and cancels the stress work.  This proves the defect-work identity.  The
pressure-free reduction, the pointwise Cauchy--Schwarz estimate, and the
time-window estimate are the same as in Theorem~\ref{thm:main-obstruction},
with \(G(A)\) replaced by \(G(C)=D\Phi(C)\).
\end{proof}

\begin{remark}[Commutation hypothesis]
\label{rem:commutation-hypothesis}
The identity \(G(C)C=T(C)\) is the compatibility relation that cancels stress
work.  In the isotropic closures used below, \(G(C)\) and \(C\) commute because
both are functions of \(C\).  For a non-isotropic closure, this commutation is
part of the entropy-dual structure.
\end{remark}

\begin{corollary}[Hookean Oldroyd--B closure]
\label{cor:hookean-closure}
For
\[
  \calD=\Spp^d,\qquad
  \Phi(C)=\tr C-\log\det C-d,\qquad
  T(C)=C-\Id,\qquad
  \mathcal R(C)=\lambda^{-1}(C-\Id),
\]
one has \(G(C)=\Id-C^{-1}\), \(G(C)C=T(C)\), and
\[
  G(C):\mathcal R(C)
  =
  \lambda^{-1}\tr(C+C^{-1}-2\Id)\ge0.
\]
Thus Theorem~\ref{thm:entropy-dual-closure} recovers the Oldroyd--B
residual-work law.
\end{corollary}

With \(b>d\), set
\[
  \calD_b=\{C\in\Spp^d:\tr C<b\},\qquad
  f_b(C)=\frac{b-d}{b-\tr C},
\]
and use the Peterlin stress
\[
  T_b(C)=f_b(C)C-\Id .
\]
The finite-extension boundary \(\tr C=b\) is now part of the admissibility
geometry, not a coordinate artifact; see \cite{Bird,ChilcottRallison1988} for
the classical FENE-P modeling background.

\begin{definition}[FENE-P entropy lever]
\label{def:fenep-lever}
For \(C\in\calD_b\), define
\[
  \Phi_b(C)
  =
  - (b-d)\log\!\left(1-\frac{\tr C}{b}\right)-\log\det C,
\]
and
\[
  G_b(C)=D\Phi_b(C)
  =
  \frac{b-d}{b-\tr C}\Id-C^{-1}.
\]
For \(K<\infty\) and \(\delta>0\), the compact FENE-P envelope is
\[
  \calK_{K,\delta}^{\rm FENE}
  =
  \{C\in\calD_b:\norm{\Log C}_{L^\infty}\le K,\ b-\tr C\ge\delta\}.
\]
\end{definition}

\begin{proposition}[Finite-extensibility boundary and lever geometry]
\label{prop:fenep-boundary-geometry}
Fix \(m>0\).  On the set
\[
  \{C\in\calD_b:\ C\ge mI\},
\]
the FENE-P lever satisfies
\[
  \norm{G_b(C)}
  \ge
  \sqrt d\,\frac{b-d}{b-\tr C}-\frac{\sqrt d}{m}.
  \label{eq:fenep-lever-blowup}
\]
Consequently, as \(\tr C\uparrow b\) with \(C\ge mI\), the residual-work lever
blows up in the isotropic direction.  Moreover, on compact envelopes
\(\calK_{K,\delta}^{\rm FENE}\), the map \(C\mapsto G_b(C)\) is Lipschitz, but
the derivative contains the singular factor \((b-\tr C)^{-2}\).  Hence no
Lipschitz bound independent of the distance to the finite-extensibility
boundary can hold.  Thus the FENE-P residual geometry is not obtained from the
Hookean one by a harmless substitution: near the finite-extensibility boundary,
the admissible residual-work fibers acquire a singular isotropic lever and
stability constants degenerate.
\end{proposition}

\begin{proof}
Since
\[
  G_b(C)=\frac{b-d}{b-\tr C}I-C^{-1},
\]
the triangle inequality gives
\[
  \norm{G_b(C)}
  \ge
  \frac{b-d}{b-\tr C}\norm{I}-\norm{C^{-1}} .
\]
If \(C\ge mI\), then \(\norm{C^{-1}}\le\sqrt d/m\), and
\(\norm{I}=\sqrt d\), proving \eqref{eq:fenep-lever-blowup}.  On
\(\calK_{K,\delta}^{\rm FENE}\), \(C^{-1}\) and its derivative are bounded by
constants depending on \(K\), while
\[
  D\left(\frac{b-d}{b-\tr C}\right)[H]
  =
  \frac{b-d}{(b-\tr C)^2}\tr H .
\]
Hence \(C\mapsto G_b(C)\) is Lipschitz with constants bounded by
\(C_{K,b}\delta^{-2}\).  The same derivative formula shows that no
boundary-independent Lipschitz constant can hold as \(\tr C=b\) is approached
in directions with nonzero trace.
\end{proof}

\begin{corollary}[FENE-P residual-work criterion]
\label{thm:fenep-residual}
For the FENE-P closure on \(\calD_b\),
\[
  \Phi_b(C)=
  - (b-d)\log\!\left(1-\frac{\tr C}{b}\right)-\log\det C,
  \qquad
  T_b(C)=f_b(C)C-\Id,
\]
and \(\mathcal R_b(C)=\lambda^{-1}T_b(C)\), the compatibility relation
\[
  G_b(C)C=T_b(C)
\]
holds with
\[
  G_b(C)=D\Phi_b(C)=f_b(C)\Id-C^{-1}.
\]
Let \((u,p,C)\) be a smooth FENE-P residual-matching profile on a time interval
\(I\), and assume that \(C(t,x)\in\calD_b\).  Let \(P(t)\) be the same
pressure-free work as in Definition~\ref{def:pressure-free-work}, but with
\(\alpha\divv(A-I)\) replaced by \(\alpha\divv T_b(C)\).  Define
\[
  L_E^{\rm FENE}=\norm{G_b(C)}_{L^2(E\times\T^d)},\qquad
  S_E^{\rm FENE}=\norm{S}_{L^2(E\times\T^d)},
\]
and
\[
  \eta_E^{\rm FENE}
  =
  \operatorname*{ess\,sup}_{t\in E}
  \frac{\left[-\int_{\T^d}G_b(C(t)):S(t)\,dx\right]_+}
       {\norm{G_b(C(t))}_{L^2}\norm{S(t)}_{L^2}},
\]
with value \(0\) when the denominator vanishes.  If the FENE-P residual state
is energy--entropy admissible, then on every measurable \(E\subset I\),
\[
  W_E\le
  \frac{\alpha}{2}
  \eta_E^{\rm FENE}L_E^{\rm FENE}S_E^{\rm FENE}.
  \label{eq:fenep-window}
\]
For each fixed time, the sharp paying residual is aligned with
\(-G_b(C)\):
\[
  S_{\min}^{\rm FENE}
  =
  -\frac{2P(t)}{\alpha\norm{G_b(C(t))}_{L^2}^2}G_b(C(t)),
  \qquad P(t)>0 .
\]
\end{corollary}

\begin{proof}
The formula for \(G_b\) is Definition~\ref{def:fenep-lever}.  The relation
\(G_b(C)C=T_b(C)\) follows immediately.  Moreover
\[
  G_b(C):T_b(C)=\tr\bigl(C^{-1}T_b(C)^2\bigr)\ge0,
\]
because \(G_b(C)=T_b(C)C^{-1}\) and \(C\) commutes with \(T_b(C)\).  The
windowed estimate and sharp minimizer are then Theorem~\ref{thm:entropy-dual-closure}.
\end{proof}

\begin{remark}[Finite extensibility in the residual-work law]
The FENE-P criterion differs from the Hookean one in two places.  First,
admissibility requires the finite-extension margin \(b-\tr C>0\).  Second, the
entropy-dual lever is
\[
  G_b(C)=f_b(C)I-C^{-1}.
\]
On compact FENE-P envelopes \(\calK_{K,\delta}^{\rm FENE}\), this is a smooth
finite-extensibility analogue of the Hookean lever.  As \(b\to\infty\) away
from the FENE boundary, \(f_b(C)\to1\) and \(G_b(C)\to I-C^{-1}\), recovering
the Oldroyd--B residual-work law.
\end{remark}

\begin{corollary}[FENE-P structured-family obstruction]
\label{cor:fenep-structured}
For a family of smooth FENE-P residual profiles on a time window \(E\), suppose
\[
  W_E^\varepsilon\ge c_W\varepsilon^{-w},\quad
  L_{E,{\rm FENE}}^\varepsilon\le C_L\varepsilon^{-\ell},\quad
  S_{E,{\rm FENE}}^\varepsilon\le C_S\varepsilon^{-r},\quad
  \eta_{E,{\rm FENE}}^\varepsilon\le C_\eta\varepsilon^{-q}.
\]
If \(w>\ell+r+q\), then the family is not FENE-P
energy--entropy admissible for all sufficiently small \(\varepsilon\).  The
same conclusion holds at equality of exponents if
\[
  c_W>\frac{\alpha}{2}C_\eta C_LC_S .
\]
\end{corollary}

\begin{proof}
This is the bulk power-law test applied to
\eqref{eq:fenep-window}.
\end{proof}

\section{Closedness, Strong Obstructions, and Weak-Compactness Interface}
\label{sec:closed-realization}

The local corrector in Section~\ref{sec:local-corrector} shows that the
least-cost paying direction is compatible with the positive cone in the
conformation channel.  It does not, by itself, construct a global solution
sequence.  The next results record the complementary obstruction: any
realization procedure whose residual data converge strongly must preserve the
signed-work budget, and any profile with a positive defect gap has a
quantitative distance from the admissible residual class.  The topology is part
of the statement: it is exactly the topology in which the three budget variables
\(P\), \(G\), and \(S\) and their product pairing are stable.  We make no claim
that weaker convergence notions preserve the same signed work.

\begin{definition}[Strong residual-data realization]
\label{def:strong-residual-realization}
Let \((u,p,B)\) be a smooth residual-matching profile on a measurable time
window \(E\), and write
\[
  A=e^B,\qquad G=I-A^{-1},\qquad S=\mathcal S[u,B],
\]
with pressure-free work \(P(t)\).  We say that \((u,p,B)\) is strongly
residual-data realizable on \(E\) by admissible positive-cone Reynolds states if
there are smooth residual-matching profiles \((u_n,p_n,B_n)\), with
\[
  A_n=e^{B_n},\qquad G_n=I-A_n^{-1},\qquad S_n=\mathcal S[u_n,B_n],
\]
and pressure-free works \(P_n(t)\), such that the associated residual states are
energy--entropy admissible and
\[
  P_n\to P\quad\hbox{in }L^1(E),
  \qquad
  G_n\to G,\quad S_n\to S
  \quad\hbox{in }L^2(E\times\T^d).
\]
\end{definition}

\begin{theorem}[Closedness of residual admissibility]
\label{thm:closed-realization}
Let \((u,p,B)\) be strongly residual-data realizable on \(E\) in the sense of
Definition~\ref{def:strong-residual-realization}.  Then
\begin{equation}
  P(t)+\frac{\alpha}{2}\int_{\T^d}G(t):S(t)\,dx\le0
  \qquad\hbox{for a.e. }t\in E .
  \label{eq:closed-realization-pointwise}
\end{equation}
Consequently, for every measurable \(F\subset E\),
\begin{equation}
  W_F\le \frac{\alpha}{2}\eta_F L_FS_F .
  \label{eq:closed-realization-window}
\end{equation}
\end{theorem}

\begin{proof}
For the approximating profiles, energy--entropy admissibility and the
pressure-free reduction give
\begin{equation}
  P_n(t)+\frac{\alpha}{2}\int_{\T^d}G_n(t):S_n(t)\,dx\le0
  \qquad\hbox{for a.e. }t\in E .
  \label{eq:approx-admissible}
\end{equation}
Set
\[
  K_n(t)=\int_{\T^d}G_n(t):S_n(t)\,dx,\qquad
  K(t)=\int_{\T^d}G(t):S(t)\,dx .
\]
The strong \(L^2\) convergence gives \(K_n\to K\) in \(L^1(E)\), because
\[
\begin{aligned}
  \norm{K_n-K}_{L^1(E)}
  &\le
  \norm{G_n-G}_{L^2(E\times\T^d)}\norm{S_n}_{L^2(E\times\T^d)}\\
  &\quad
  +\norm{G}_{L^2(E\times\T^d)}
   \norm{S_n-S}_{L^2(E\times\T^d)} .
\end{aligned}
\]
Here \(\{S_n\}\) is bounded in \(L^2\) by the assumed convergence.  Since
\(P_n\to P\) in \(L^1(E)\), the left-hand side of
\eqref{eq:approx-admissible} converges in \(L^1(E)\) to
\[
  P(t)+\frac{\alpha}{2}K(t).
\]
The cone of non-positive \(L^1\) functions is closed: if a limit were positive
on a set of positive measure, its positive part would have positive integral,
contradicting \(L^1\) convergence from non-positive functions.  This proves
\eqref{eq:closed-realization-pointwise}.  Applying
Theorem~\ref{thm:main-obstruction} to the limiting residual data on the
subwindow \(F\) gives \eqref{eq:closed-realization-window}.
\end{proof}

\begin{theorem}[Sharpness of the residual-data topology]
\label{thm:weak-sharpness}
Let \(E\) be a time window of positive finite measure.

\emph{(i) One-sided strong closure.}
Let \(P_n\to P\) in \(L^1(E)\), and suppose that
\((P_n,G_n,S_n)\) are admissible residual data, namely
\[
  P_n(t)+\frac{\alpha}{2}\int_{\T^d}G_n(t):S_n(t)\,dx\le0
  \qquad\hbox{for a.e. }t\in E .
\]
If either
\[
  G_n\to G\hbox{ strongly in }L^2(E\times\T^d),
  \qquad
  S_n\rightharpoonup S\hbox{ weakly in }L^2(E\times\T^d),
\]
or the same statement holds with the roles of \(G_n\) and \(S_n\) interchanged,
then \((P,G,S)\) is admissible:
\[
  P(t)+\frac{\alpha}{2}\int_{\T^d}G(t):S(t)\,dx\le0
  \qquad\hbox{for a.e. }t\in E .
\]

\emph{(ii) Failure of weak--weak closure.}
Fix \(p_0>0\).  There are smooth residual data
\((P_n,G_n,S_n)\) such that
\[
  P_n\to p_0\quad\hbox{strongly in }L^1(E),
  \qquad
  G_n\rightharpoonup0,\quad S_n\rightharpoonup0
  \quad\hbox{weakly in }L^2(E\times\T^d),
\]
each \((P_n,G_n,S_n)\) is admissible, but the weak limit
\((p_0,0,0)\) is not admissible.  Thus the signed residual-work budget is not
closed under weak convergence of both product factors unless an additional
defect measure, or an equivalent compactness assumption, is retained.
\end{theorem}

\begin{proof}
We prove the first assertion when \(G_n\to G\) strongly and
\(S_n\rightharpoonup S\) weakly; the other case is identical.  Let
\(\varphi\in L^\infty(E)\), \(\varphi\ge0\).  Since
\(\varphi G_n\to\varphi G\) strongly in \(L^2(E\times\T^d)\), weak convergence
of \(S_n\) gives
\[
  \int_E\varphi(t)\int_{\T^d}G_n(t):S_n(t)\,dx\,dt
  \longrightarrow
  \int_E\varphi(t)\int_{\T^d}G(t):S(t)\,dx\,dt .
\]
Together with \(P_n\to P\) in \(L^1(E)\), this permits passage to the limit in
the admissibility inequality tested against every non-negative \(\varphi\):
\[
  \int_E\varphi(t)
  \left[
    P(t)+\frac{\alpha}{2}\int_{\T^d}G(t):S(t)\,dx
  \right]dt
  \le0 .
\]
Hence the bracket is non-positive a.e. on \(E\).

For the second assertion, choose a smooth sequence
\(\Psi_n:\T^d\to\Sym^d\) with
\[
  \norm{\Psi_n}_{L^2(\T^d)}=1,
  \qquad
  \Psi_n\rightharpoonup0\quad\hbox{weakly in }L^2(\T^d;\Sym^d).
\]
For instance, one may take normalized Fourier modes in a fixed nonzero
symmetric direction.  Set
\[
  P_n(t)=p_0,\qquad
  G_n(t,x)=\Psi_n(x),\qquad
  S_n(t,x)=-\frac{2p_0}{\alpha}\Psi_n(x).
\]
Then \(P_n\to p_0\) strongly in \(L^1(E)\), while
\(G_n\rightharpoonup0\) and \(S_n\rightharpoonup0\) weakly in
\(L^2(E\times\T^d)\).  For every \(t\),
\[
  P_n(t)+\frac{\alpha}{2}\int_{\T^d}G_n(t):S_n(t)\,dx
  =
  p_0-\frac{\alpha}{2}\frac{2p_0}{\alpha}\norm{\Psi_n}_{L^2}^2
  =0 .
\]
Thus every approximating triple is admissible.  The weak limit is
\((p_0,0,0)\), for which the signed defect equals \(p_0>0\).  This proves the
failure of weak--weak closure.
\end{proof}

\begin{remark}[Product defect]
The preceding theorem identifies the compactness threshold in the residual
variables themselves.  Weak limits of \(G_n\) and \(S_n\) do not determine the
limit of the product \(G_n:S_n\).  A weaker realization theory would therefore
need to carry an additional product-defect measure, analogous to a Reynolds or
oscillation defect, before the signed work balance could be closed.
\end{remark}

\begin{theorem}[Defect-measure compactification of weak residual limits]
\label{thm:defect-compactification}
Let \(P_n\to P\) in \(L^1(E)\), and let
\[
  G_n\rightharpoonup G,\qquad S_n\rightharpoonup S
  \quad\hbox{weakly in }L^2(E\times\T^d;\Sym^d).
\]
Set
\[
  K_n(t)=\int_{\T^d}G_n(t):S_n(t)\,dx,
  \qquad
  K(t)=\int_{\T^d}G(t):S(t)\,dx .
\]
Assume that the signed measures \(K_n(t)\,dt\) converge weakly-* in
\(\mathcal M(E)\) to \(K(t)\,dt+\mu\), where \(\mu\) is a finite signed measure
on \(E\).  If every residual triple \((P_n,G_n,S_n)\) is admissible, then
\begin{equation}
  P\,dt+\frac{\alpha}{2}K\,dt+\frac{\alpha}{2}\mu\le0
  \qquad\hbox{in }\mathcal M(E).
  \label{eq:defect-augmented-budget}
\end{equation}
In particular, weak residual limits are closed after the product defect
\(\mu\) is retained.

Conversely, for every smooth density \(m(t)\) on \(E\), there are smooth
weakly null fields \(G_n,S_n\) in \(L^2(E\times\T^d;\Sym^d)\) such that
\[
  \left(\int_{\T^d}G_n(t):S_n(t)\,dx\right)dt
  \stackrel{*}{\rightharpoonup} m(t)\,dt .
\]
Thus the defect variable in \eqref{eq:defect-augmented-budget} is not
determined by the weak limits of \(G_n\) and \(S_n\).
\end{theorem}

\begin{proof}
For each \(n\), admissibility gives
\[
  P_n(t)\,dt+\frac{\alpha}{2}K_n(t)\,dt\le0
  \qquad\hbox{as a signed measure on }E .
\]
Let \(\varphi\in C_c(E)\), \(\varphi\ge0\).  Since
\(P_n\to P\) in \(L^1(E)\) and \(K_n\,dt\stackrel{*}{\rightharpoonup}K\,dt+\mu\),
we may pass to the limit in
\[
  \int_E\varphi P_n\,dt+\frac{\alpha}{2}\int_E\varphi K_n\,dt\le0 .
\]
This gives
\[
  \int_E\varphi P\,dt
  +\frac{\alpha}{2}\int_E\varphi K\,dt
  +\frac{\alpha}{2}\int_E\varphi\,d\mu\le0 ,
\]
which is exactly \eqref{eq:defect-augmented-budget}.

For the realization of smooth defect densities, choose a fixed
\(H\in\Sym^d\) with \(|H|=1\), and choose smooth functions
\(\varphi_n:\T^d\to\mathbb R\) such that
\[
  \varphi_n\rightharpoonup0\quad\hbox{weakly in }L^2(\T^d),
  \qquad
  \int_{\T^d}\varphi_n^2\,dx=1 .
\]
For example, one may take normalized high-frequency trigonometric modes.  Set
\[
  G_n(t,x)=\varphi_n(x)H,\qquad
  S_n(t,x)=m(t)\varphi_n(x)H .
\]
Then \(G_n\rightharpoonup0\) and \(S_n\rightharpoonup0\) weakly in
\(L^2(E\times\T^d)\), while
\[
  \int_{\T^d}G_n(t):S_n(t)\,dx
  =
  m(t)\int_{\T^d}\varphi_n^2\,dx\, |H|^2
  =
  m(t).
\]
This proves the claim.
\end{proof}

\begin{theorem}[Irreducibility of the entropy-dual product defect]
\label{thm:irreducible-product-defect}
Let \(E\) be a compact time interval.  For any two smooth densities
\(m_1,m_2\in C^\infty(E)\), there are two smooth sequences
\[
  (G_n^{(j)},S_n^{(j)}),\qquad j=1,2,
\]
bounded in \(L^2(E\times\T^d;\Sym^d)\), such that
\[
  G_n^{(j)}\rightharpoonup0,\qquad
  S_n^{(j)}\rightharpoonup0
  \quad\hbox{weakly in }L^2(E\times\T^d;\Sym^d),
\]
but
\[
  \left(\int_{\T^d}G_n^{(j)}(t):S_n^{(j)}(t)\,dx\right)dt
  \stackrel{*}{\rightharpoonup}m_j(t)\,dt .
\]
Consequently the product-defect measure in
Theorem~\ref{thm:defect-compactification} is not determined by the marginal
weak limits of \(G_n\) and \(S_n\), even in smooth bounded sequences.  Any
weak residual compactness theorem that keeps only the marginal weak limits
\((P,G,S)\) necessarily loses signed residual work.
\end{theorem}

\begin{proof}
Apply the construction in the second part of
Theorem~\ref{thm:defect-compactification} separately with \(m=m_1\) and
\(m=m_2\).  Both sequences have the same marginal weak limits, namely
\((0,0)\), but their product measures converge to different limits.  Hence no
functional of the marginal weak limits can determine the product defect.
\end{proof}

\begin{corollary}[Local patching versus global defect accounting]
\label{cor:local-global-defect-accounting}
Let \((P_n,G_n,S_n)\) be admissible residual data on a fixed time window \(E\).
The triples may be produced by local positive-cone corrections, localized
packets, or any other smooth residual construction.  Suppose
\[
  P_n\to P\quad\hbox{in }L^1(E),\qquad
  G_n\rightharpoonup G,\quad S_n\rightharpoonup S
  \quad\hbox{weakly in }L^2(E\times\T^d),
\]
and suppose that
\[
  \left(\int_{\T^d}G_n:S_n\,dx\right)dt
  \stackrel{*}{\rightharpoonup}
  \left(\int_{\T^d}G:S\,dx\right)dt+\mu .
\]
Then the global limit satisfies the augmented budget
\begin{equation}
  P\,dt+\frac{\alpha}{2}
  \left(\int_{\T^d}G:S\,dx\right)dt
  +\frac{\alpha}{2}\mu\le0 .
  \label{eq:local-global-augmented-budget}
\end{equation}
If, in addition, one of the two factors converges strongly in
\(L^2(E\times\T^d)\), then \(\mu=0\) and the unaugmented signed-work budget is
closed.  Hence a sequence of many small local corrections can bypass a strong
residual-data obstruction only by losing strong compactness and by carrying the
lost work as an explicit product-defect measure.
\end{corollary}

\begin{proof}
The augmented inequality is exactly Theorem~\ref{thm:defect-compactification}.
If \(G_n\to G\) strongly and \(S_n\rightharpoonup S\) weakly, then for every
\(\varphi\in C_c(E)\),
\[
  \int_E\varphi(t)\int_{\T^d}G_n:S_n\,dx\,dt
  \to
  \int_E\varphi(t)\int_{\T^d}G:S\,dx\,dt ,
\]
because \(\varphi G_n\to\varphi G\) strongly in \(L^2(E\times\T^d)\).  Thus
\(\mu=0\).  The case in which \(S_n\to S\) strongly is identical.
\end{proof}

\begin{corollary}[Defect charge forced by global violation]
\label{cor:defect-charge-violation}
In the setting of Corollary~\ref{cor:local-global-defect-accounting}, set
\[
  \Lambda
  =
  P\,dt+\frac{\alpha}{2}
  \left(\int_{\T^d}G:S\,dx\right)dt .
\]
Then every Borel set \(A\subset E\) with \(\Lambda(A)>0\) satisfies
\[
  \mu(A)\le -\frac{2}{\alpha}\Lambda(A).
\]
In particular, if \(\Lambda_+(E)>0\), where \(\Lambda_+\) is the positive part
in the Hahn decomposition of \(\Lambda\), then
\[
  \mu_-(E)\ge \frac{2}{\alpha}\Lambda_+(E).
\]
Thus a globally violating residual profile cannot be produced by patching
local positive-cone corrections with zero product defect.  Any weak patching
that reaches such a profile must store at least the displayed amount of
negative product defect.
\end{corollary}

\begin{proof}
Corollary~\ref{cor:local-global-defect-accounting} gives
\[
  \Lambda+\frac{\alpha}{2}\mu\le0
\]
as signed measures.  Testing this inequality on a Borel set \(A\) gives the
first claim.  If \(H\) is a Hahn positive set for \(\Lambda\), then
\(\Lambda(H)=\Lambda_+(E)\).  Applying the first claim to \(H\) yields
\(\mu(H)\le -(2/\alpha)\Lambda_+(E)\), and hence
\(\mu_-(E)\ge \mu_-(H)\ge (2/\alpha)\Lambda_+(E)\).
\end{proof}

\begin{theorem}[Strict converse for positive-cone residual data]
\label{thm:strict-augmented-converse}
Let \(E\) be a compact time interval.  Let
\[
  P\in C^\infty(E),\qquad
  G,S\in C^\infty(E\times\T^d;\Sym^d),
  \qquad
  m\in C^\infty(E),
\]
assume that \(I-G(t,x)\ge\kappa I\) on \(E\times\T^d\) for some
\(\kappa>0\), and suppose that the augmented residual budget has strict slack:
\begin{equation}
  P(t)+\frac{\alpha}{2}\int_{\T^d}G(t):S(t)\,dx
  +\frac{\alpha}{2}m(t)\le-\sigma
  \qquad\hbox{on }E
  \label{eq:strict-augmented-slack}
\end{equation}
for some \(\sigma>0\).  Then there are smooth residual triples
\[
  (P_n,G_n,S_n)
\]
such that \(P_n=P\), \(I-G_n\ge(\kappa/2)I\), and therefore
\(G_n=I-A_n^{-1}\) for smooth positive-definite
\[
  A_n=(I-G_n)^{-1},
\]
\[
  G_n\rightharpoonup G,\qquad S_n\rightharpoonup S
  \quad\hbox{weakly in }L^2(E\times\T^d;\Sym^d),
\]
and
\begin{equation}
  \left(\int_{\T^d}G_n(t):S_n(t)\,dx\right)dt
  \stackrel{*}{\rightharpoonup}
  \left(\int_{\T^d}G(t):S(t)\,dx\right)dt+m(t)\,dt
  \label{eq:strict-converse-product}
\end{equation}
as signed measures on \(E\).  Moreover, for all sufficiently large \(n\),
\begin{equation}
  P_n(t)+\frac{\alpha}{2}\int_{\T^d}G_n(t):S_n(t)\,dx
  \le -\frac{\sigma}{2}
  \qquad\hbox{on }E .
  \label{eq:strict-converse-admissible}
\end{equation}
Thus every smooth augmented budget with strict negative margin is realized by
classically admissible localized positive-cone residual packets.  Moreover, for
any prescribed non-empty open set \(U\subset\T^d\), the construction can be made
so that \(G_n-G\) and \(S_n-S\) are supported in \(E\times U\).
\end{theorem}

\begin{proof}
Choose a fixed matrix \(H\in\Sym^d\) with \(|H|=1\).  Fix a non-empty open set
\(U\subset\T^d\); taking \(U=\T^d\) gives the global version.  Let
\(\varphi_n\in C^\infty_c(U)\) be normalized high-frequency scalar modes such
that
\[
  \varphi_n\rightharpoonup0\quad\hbox{weakly in }L^2(\T^d),
  \qquad
  \int_{\T^d}\varphi_n^2\,dx=1,
\]
and such that, for every smooth \(f(t,x)\), the spatial integrals
\[
  \int_{\T^d}f(t,x)\varphi_n(x)\,dx
\]
converge to zero uniformly for \(t\in E\).  Trigonometric modes give such a
sequence when \(U=\T^d\).  For general \(U\), take a non-zero
\(\chi\in C^\infty_c(U)\), multiply by high-frequency trigonometric modes, and
normalize; the Riemann--Lebesgue lemma gives the same weak convergence, and the
sequence has a uniform \(L^\infty\) bound after normalization.  Fix
\(\varepsilon>0\) so small that
\[
  \varepsilon\sup_n\norm{\varphi_n}_{L^\infty(\T^d)}\le\frac{\kappa}{2}.
\]
Define
\[
  \widehat G_n(t,x)=\varepsilon\varphi_n(x)H,
  \qquad
  \widehat S_n(t,x)=\frac{m(t)}{\varepsilon}\varphi_n(x)H,
\]
and set
\[
  P_n=P,\qquad
  G_n=G+\widehat G_n,\qquad
  S_n=S+\widehat S_n .
\]
The weak convergence of \(G_n\) and \(S_n\) to \(G\) and \(S\) is immediate.
Moreover,
\[
  I-G_n=(I-G)-\varepsilon\varphi_n H\ge\frac{\kappa}{2}I,
\]
so \(G_n\) is an entropy lever associated with the positive-definite tensor
\(A_n=(I-G_n)^{-1}\).

The product integral is
\[
\begin{aligned}
  \int_{\T^d}G_n:S_n\,dx
  &=
  \int_{\T^d}G:S\,dx
  +\frac{m(t)}{\varepsilon}\int_{\T^d}\varphi_n\,G:H\,dx  \\
  &\quad
  +\varepsilon\int_{\T^d}\varphi_n\,H:S\,dx
  +m(t)\int_{\T^d}\varphi_n^2 |H|^2\,dx .
\end{aligned}
\]
The last term is exactly \(m(t)\).  The two cross terms converge to zero
uniformly on \(E\).  Hence
\[
  \int_{\T^d}G_n:S_n\,dx
  =
  \int_{\T^d}G:S\,dx+m(t)+r_n(t),
  \qquad
  \norm{r_n}_{L^\infty(E)}\to0 .
\]
This proves \eqref{eq:strict-converse-product}.  Combining the last identity
with the strict slack assumption \eqref{eq:strict-augmented-slack} gives
\[
  P_n(t)+\frac{\alpha}{2}\int_{\T^d}G_n:S_n\,dx
  \le
  -\sigma+\frac{\alpha}{2}\norm{r_n}_{L^\infty(E)} .
\]
For all sufficiently large \(n\), the right-hand side is bounded above by
\(-\sigma/2\), which proves \eqref{eq:strict-converse-admissible}.
\end{proof}

\begin{corollary}[Quantitative localized packet bounds]
\label{cor:quantitative-localized-packets}
In Theorem~\ref{thm:strict-augmented-converse}, fix a non-empty open
\(U\subset\T^d\).  For every sufficiently small \(\eta>0\), the packet sequence
can be chosen so that, for a constant \(C_U\) depending only on the spatial
cutoff used in \(U\),
\begin{align}
  \operatorname{supp}(G_n-G)\cup\operatorname{supp}(S_n-S)
  &\subset E\times U, \label{eq:packet-support}\\
  \norm{G_n-G}_{L^\infty(E\times\T^d)}
  &\le \eta, \label{eq:packet-small-lever}\\
  \norm{S_n-S}_{L^2(E\times\T^d)}
  &\le C_U\eta^{-1}\norm{m}_{L^2(E)}, \label{eq:packet-S-cost}
\end{align}
and
\begin{equation}
  \int_{\T^d}G_n(t):S_n(t)\,dx
  =
  \int_{\T^d}G(t):S(t)\,dx+m(t)+r_n(t),
  \qquad
  \norm{r_n}_{L^\infty(E)}\to0 .
  \label{eq:packet-uniform-product-error}
\end{equation}
If \(m\) is supported in a compact subinterval \(J\Subset E\), then the packets
can be chosen with support in \(J'\times U\) for any open
\(J'\Subset E\) containing \(\operatorname{supp}m\).
\end{corollary}

\begin{proof}
Use the construction in the proof of
Theorem~\ref{thm:strict-augmented-converse}, and choose
\[
  \varepsilon=\eta\bigl(\sup_n\norm{\varphi_n}_{L^\infty}\bigr)^{-1}.
\]
Then \(\widehat G_n=\varepsilon\varphi_nH\) satisfies
\eqref{eq:packet-small-lever}, while
\[
  \widehat S_n=\frac{m(t)}{\varepsilon}\varphi_nH
\]
gives \eqref{eq:packet-S-cost} after absorbing the uniform
\(L^\infty\)-bound for \(\varphi_n\) and the fixed normalization into \(C_U\).
The support property follows from \(\varphi_n\in C_c^\infty(U)\), and
\eqref{eq:packet-uniform-product-error} is exactly the product computation in
the theorem.  If \(m\) is supported in \(J\), multiply \(\widehat G_n\) by a
time cutoff equal to one on \(\operatorname{supp}m\) and supported in \(J'\);
the product term \(m(t)\) is unchanged.
\end{proof}

\begin{proposition}[Optimality of the reciprocal packet cost]
\label{prop:reciprocal-packet-cost}
Let \(U\subset\T^d\) be non-empty and open, let \(\eta>0\), and let
\(\delta G_n,\delta S_n\in L^2(E\times\T^d;\Sym^d)\) be supported in
\(E\times U\).  Assume that
\[
  \norm{\delta G_n}_{L^\infty(E\times\T^d)}\le\eta
\]
and that the self-interaction of the packet produces a density \(m\in L^2(E)\):
\[
  \int_{\T^d}\delta G_n(t):\delta S_n(t)\,dx
  \longrightarrow m(t)
  \qquad\hbox{in }L^2(E).
\]
Then
\begin{equation}
  \liminf_{n\to\infty}
  \norm{\delta S_n}_{L^2(E\times\T^d)}
  \ge
  |U|^{-1/2}\eta^{-1}\norm{m}_{L^2(E)} .
  \label{eq:reciprocal-packet-cost-lower}
\end{equation}
Here \(|U|\) denotes the Lebesgue measure of \(U\).  Consequently, the
\(\eta^{-1}\) dependence in
Corollary~\ref{cor:quantitative-localized-packets} is forced by any localized
small-lever packet whose limiting product defect is carried by
\(\delta G_n:\delta S_n\).
\end{proposition}

\begin{proof}
For a.e. \(t\in E\), Cauchy--Schwarz and the support condition give
\[
  \left|\int_{\T^d}\delta G_n(t):\delta S_n(t)\,dx\right|
  \le
  \norm{\delta G_n(t)}_{L^2_x(U)}
  \norm{\delta S_n(t)}_{L^2_x(U)}
  \le
  \eta |U|^{1/2}\norm{\delta S_n(t)}_{L^2_x(U)} .
\]
Taking the \(L^2(E)\)-norm and passing to the lower limit yields
\eqref{eq:reciprocal-packet-cost-lower}.
\end{proof}

\begin{remark}[Residual-data converse and lifting interface]
Theorem~\ref{thm:strict-augmented-converse} is a genuine converse to the
defect-measure compactification at the level of residual data.  It shows that,
once the product-defect measure is retained, the augmented budget is not merely
closed but also realizable by localized admissible positive-cone packets under
a strict margin, with the quantitative bounds of
Corollary~\ref{cor:quantitative-localized-packets}; Proposition~\ref{prop:reciprocal-packet-cost}
shows that the reciprocal residual-cost scale is not an artifact of the
construction.  The theorem still does not lift those packets to exact
Oldroyd--B weak solutions.  The remaining step is a positive-cone
weak-solution lifting problem: one must construct oscillations satisfying the
momentum equation, the conformation transport-stretch equation, positivity of
\(A\), and the energy--entropy inequality simultaneously.
\end{remark}

\begin{definition}[Residual defect gap]
\label{def:residual-defect-gap}
For residual data
\[
  P\in L^1(E),\qquad G,S\in L^2(E\times\T^d;\Sym^d),
\]
define
\[
  \Delta_{P,G,S}(t)
  =
  P(t)+\frac{\alpha}{2}\int_{\T^d}G(t):S(t)\,dx
\]
and
\begin{equation}
  \mathfrak D_E(P,G,S)
  =
  \norm{[\Delta_{P,G,S}]_+}_{L^1(E)}.
  \label{eq:defect-gap}
\end{equation}
For a smooth residual-matching profile, we write
\[
  \mathfrak D_E[u,B]
  =
  \mathfrak D_E(P,I-e^{-B},\mathcal S[u,B]).
\]
\end{definition}

\begin{theorem}[Quantitative separation from admissible residual data]
\label{thm:quantitative-gap}
Let \(P,\widetilde P\in L^1(E)\) and
\[
  G,S,\widetilde G,\widetilde S\in L^2(E\times\T^d;\Sym^d).
\]
Assume that \((\widetilde P,\widetilde G,\widetilde S)\) is admissible in the
residual sense,
\begin{equation}
  \widetilde P(t)+\frac{\alpha}{2}
  \int_{\T^d}\widetilde G(t):\widetilde S(t)\,dx\le0
  \qquad\hbox{for a.e. }t\in E .
  \label{eq:tilde-admissible}
\end{equation}
Then
\begin{equation}
\begin{aligned}
  \mathfrak D_E(P,G,S)
  &\le
  \norm{P-\widetilde P}_{L^1(E)}
  +\frac{\alpha}{2}
   \norm{G-\widetilde G}_{L^2(E\times\T^d)}
   \norm{S}_{L^2(E\times\T^d)}\\
  &\quad
  +\frac{\alpha}{2}
   \norm{\widetilde G}_{L^2(E\times\T^d)}
   \norm{S-\widetilde S}_{L^2(E\times\T^d)} .
  \label{eq:quantitative-gap}
\end{aligned}
\end{equation}
In particular, if \(\mathfrak D_E(P,G,S)>0\), then every admissible residual
triple has a positive weighted distance from \((P,G,S)\) in the
\(L^1\times L^2\times L^2\) residual topology.
\end{theorem}

\begin{proof}
Let
\[
  \Delta(t)=P(t)+\frac{\alpha}{2}\int_{\T^d}G(t):S(t)\,dx
\]
and define \(\widetilde\Delta(t)\) analogously.  By
\eqref{eq:tilde-admissible}, \(\widetilde\Delta\le0\) a.e.; hence
\[
  [\Delta]_+\le |\Delta-\widetilde\Delta| .
\]
Therefore
\[
\begin{aligned}
  \mathfrak D_E(P,G,S)
  &\le
  \norm{P-\widetilde P}_{L^1(E)}
  +\frac{\alpha}{2}
  \int_E
  \left|
  \int_{\T^d}(G:S-\widetilde G:\widetilde S)\,dx
  \right|dt .
\end{aligned}
\]
Decompose
\[
  G:S-\widetilde G:\widetilde S
  =
  (G-\widetilde G):S+\widetilde G:(S-\widetilde S).
\]
Cauchy's inequality in space-time gives \eqref{eq:quantitative-gap}.
\end{proof}

\begin{proposition}[Error-bar defect-gap certificate]
\label{prop:error-bar-defect-gap}
Let \((P_h,G_h,S_h)\) be observed or discretized residual data on \(E\), and
let \((P,G,S)\) be the exact residual data.  Suppose that
\[
  \norm{P-P_h}_{L^1(E)}\le\varepsilon_P,\qquad
  \norm{G-G_h}_{L^2(E\times\T^d)}\le\varepsilon_G,\qquad
  \norm{S-S_h}_{L^2(E\times\T^d)}\le\varepsilon_S .
\]
Set
\[
  \mathcal E_h
  =
  \varepsilon_P
  +\frac{\alpha}{2}
  \left[
    \varepsilon_G\norm{S_h}_{L^2(E\times\T^d)}
    +\bigl(\norm{G_h}_{L^2(E\times\T^d)}+\varepsilon_G\bigr)
     \varepsilon_S
  \right].
\]
Then
\begin{equation}
  \left|
  \mathfrak D_E(P,G,S)-\mathfrak D_E(P_h,G_h,S_h)
  \right|
  \le\mathcal E_h .
  \label{eq:error-bar-defect-gap}
\end{equation}
Consequently, if
\[
  \mathfrak D_E(P_h,G_h,S_h)>\mathcal E_h,
\]
then the exact residual data violate the signed-work inequality on \(E\).  The
certified defect gap is at least
\[
  \mathfrak D_E(P_h,G_h,S_h)-\mathcal E_h .
\]
If the observed gap is no larger than \(\mathcal E_h\), the residual-work test
is inconclusive rather than affirmative: the error bar can hide the violation.
\end{proposition}

\begin{proof}
Let
\[
  \Delta=P+\frac{\alpha}{2}\int_{\T^d}G:S\,dx,
  \qquad
  \Delta_h=P_h+\frac{\alpha}{2}\int_{\T^d}G_h:S_h\,dx .
\]
Since \(a\mapsto[a]_+\) is \(1\)-Lipschitz,
\[
  \left|\mathfrak D_E(P,G,S)-\mathfrak D_E(P_h,G_h,S_h)\right|
  \le \norm{\Delta-\Delta_h}_{L^1(E)} .
\]
Moreover,
\[
  G:S-G_h:S_h=(G-G_h):S_h+G:(S-S_h).
\]
Therefore, by Cauchy's inequality,
\[
\begin{aligned}
  \norm{\Delta-\Delta_h}_{L^1(E)}
  &\le
  \varepsilon_P
  +\frac{\alpha}{2}
  \varepsilon_G\norm{S_h}_{L^2(E\times\T^d)}
  +\frac{\alpha}{2}
  \norm{G}_{L^2(E\times\T^d)}\varepsilon_S\\
  &\le
  \varepsilon_P
  +\frac{\alpha}{2}
  \left[
    \varepsilon_G\norm{S_h}_{L^2(E\times\T^d)}
    +\bigl(\norm{G_h}_{L^2(E\times\T^d)}+\varepsilon_G\bigr)
     \varepsilon_S
  \right],
\end{aligned}
\]
which is \eqref{eq:error-bar-defect-gap}.  The consequences follow by
subtracting the error bar from the observed defect gap.
\end{proof}

\begin{corollary}[Strong non-realization of violating residual profiles]
\label{cor:strong-nonrealization}
Let \((u,p,B)\) be a smooth residual-matching profile on \(E\).  If
\[
  P(t)+\frac{\alpha}{2}\int_{\T^d}G(t):\mathcal S[u,B](t)\,dx>0
\]
on a subset of \(E\) of positive measure, then \((u,p,B)\) is not strongly
residual-realizable on \(E\) by energy--entropy admissible positive-cone
Reynolds states.  The same conclusion holds if the windowed estimate
\[
  W_F\le \frac{\alpha}{2}\eta_F L_FS_F
\]
fails on some measurable \(F\subset E\).
\end{corollary}

\begin{proof}
Either violation contradicts Theorem~\ref{thm:closed-realization}.
\end{proof}

\begin{corollary}[Quantitative gap for the finite-thickness obstruction]
\label{cor:finite-thickness-quantitative-gap}
For the finite-thickness shear-layer profile in
Proposition~\ref{prop:concrete-layer},
\begin{equation}
  \mathfrak D_E[u^\varepsilon,B^\varepsilon]
  \ge
  |E|\nu\norm{\theta'}_{L^2(\R)}^2\,
  \varepsilon^{-(2\beta+1)} .
  \label{eq:finite-thickness-gap}
\end{equation}
Consequently, any admissible residual triple
\((\widetilde P^\varepsilon,\widetilde G^\varepsilon,
\widetilde S^\varepsilon)\) must satisfy the lower bound obtained from
\eqref{eq:quantitative-gap} with
\((P,G,S)=(P^\varepsilon,G^\varepsilon,S^\varepsilon)\) and the right-hand side
at least the quantity in \eqref{eq:finite-thickness-gap}.
\end{corollary}

\begin{proof}
For the shear-layer profile, the pressure-free work satisfies
\[
  P^\varepsilon(t)
  =
  \nu\norm{\partial_2U^\varepsilon}_{L^2(\T^2)}^2
  =
  \nu\varepsilon^{-(2\beta+1)}\norm{\theta'}_{L^2(\R)}^2 .
\]
Moreover \(G^\varepsilon\) is diagonal, while the stretching part of
\(S^\varepsilon\) is off-diagonal.  Hence only the relaxation term contributes
to \(G^\varepsilon:S^\varepsilon\), and
\[
  G^\varepsilon:S^\varepsilon
  =
  \lambda^{-1}G^\varepsilon:(A^\varepsilon-\Id)
  =
  \lambda^{-1}
  \tr\bigl(A^\varepsilon+(A^\varepsilon)^{-1}-2\Id\bigr)
  \ge0 .
\]
Thus the signed defect is at least \(P^\varepsilon(t)\), and integrating over
\(E\) gives \eqref{eq:finite-thickness-gap}.
\end{proof}

\begin{corollary}[Finite-thickness profiles are not strongly residual-data realizable]
\label{cor:finite-thickness-nonrealization}
For every sufficiently small \(\varepsilon>0\), the finite-thickness
shear-layer profile in Proposition~\ref{prop:concrete-layer} is not strongly
residual-data realizable on \(E\) by energy--entropy admissible positive-cone
Reynolds states.
\end{corollary}

\begin{proof}
Proposition~\ref{prop:concrete-layer} and the bulk power-law test show that the
profile violates the residual-work budget for all sufficiently small
\(\varepsilon\).  Corollary~\ref{cor:strong-nonrealization} then excludes strong
residual-data realization in the topology of
Definition~\ref{def:strong-residual-realization}.
\end{proof}

\begin{remark}[Topological scope of the finite-thickness obstruction]
The obstruction above is topological and residual-level.  It rules out
realization procedures whose pressure-free work, entropy lever, and conformation
residual converge strongly to a violating profile.  The topology is sharp at
the level of residual data: Theorem~\ref{thm:weak-sharpness} shows that
one-sided strong convergence is enough to pass the signed product, while
weak--weak convergence can hide an oscillatory \(G:S\) defect.  The theorem
therefore identifies the missing datum.  Theorem~\ref{thm:defect-compactification}
shows that if this product defect is retained as a signed measure, the
augmented residual budget is closed.  The strict converse in
Theorem~\ref{thm:strict-augmented-converse} shows that, at the residual-data
level, this augmented budget is also sufficient under a strict margin.  Lifting
such residual-data realizations to genuine weak solutions requires additional
compactness, oscillation, and positive-cone control beyond the residual
compactification used here.
\end{remark}

\subsection{Defect-Augmented Weak-Solution Interface and Packet Lifting}
\label{sec:weak-compactness}

The preceding sections were formulated entirely at the residual-data level.
Theorem~\ref{thm:strict-augmented-converse} proves the converse realization of
the augmented budget in that category, under strict slack.  We now isolate the
compactness class to which those residual data belong after weak limiting.  The
framework below records the opposite direction for genuine weak solutions: any
realizing sequence with enough compactness to identify its Reynolds defects
must enter the same closed augmented class.

\begin{definition}[Augmented positive-cone Reynolds state]
\label{def:augmented-reynolds-state}
Let \(E\) be a time window.  An augmented positive-cone Reynolds state on \(E\)
is a tuple
\[
  (u,p,B,R,S,\mu),
\]
where \((u,p,B,R,S)\) is a positive-cone Reynolds state in the sense of
Definition~\ref{def:positive-cone-reynolds-state} and
\(\mu\in\mathcal M(E)\) is a finite signed measure.  With
\[
  A=e^B,\qquad G=\Id-A^{-1},
\]
and with \(P\) denoting the pressure-free work of \(R\) along \(u\), define the
augmented residual-work measure
\begin{equation}
  \mathfrak B[u,B,R,S,\mu]
  :=
  P\,dt+\frac{\alpha}{2}
  \left(\int_{\T^d}G:S\,dx\right)dt
  +\frac{\alpha}{2}\mu .
  \label{eq:augmented-budget-measure}
\end{equation}
The augmented state is admissible if
\[
  \mathfrak B[u,B,R,S,\mu]\le0
  \qquad\hbox{in }\mathcal M(E).
\]
When \(\mu=0\), this is exactly the residual-work admissibility inequality for
the underlying positive-cone Reynolds state.
\end{definition}

\begin{theorem}[Sequential closedness of the augmented compactness framework]
\label{thm:closed-compactness-framework}
Let
\[
  (u_n,p_n,B_n,R_n,S_n,\mu_n)
\]
be admissible augmented positive-cone Reynolds states on \(E\).  Suppose, after
passing to a subsequence, that the convergence is strong enough to pass the
linear and nonlinear terms in the Reynolds equations, so that
\[
  u_n\to u,\qquad B_n\to B,\qquad p_n\to p,\qquad
  R_n\to R,\qquad S_n\to S
\]
in the corresponding distributional senses and \(A_n=e^{B_n}\to A=e^B\).
Assume moreover that, for \(G_n=\Id-A_n^{-1}\), \(G=\Id-A^{-1}\), and the
pressure-free works \(P_n\) and \(P\),
\begin{align}
  P_n\,dt &\stackrel{*}{\rightharpoonup} P\,dt,
  \label{eq:closed-framework-P}\\
  \left(\int_{\T^d}G_n:S_n\,dx\right)dt+\mu_n
  &\stackrel{*}{\rightharpoonup}
  \left(\int_{\T^d}G:S\,dx\right)dt+\mu
  \label{eq:closed-framework-product}
\end{align}
as finite signed measures on \(E\).  Then
\((u,p,B,R,S,\mu)\) is an admissible augmented positive-cone Reynolds state.
In particular, the class defined by
Definition~\ref{def:augmented-reynolds-state} is sequentially closed under the
convergence that retains the pressure-free work and the \(G:S\) product defect.
\end{theorem}

\begin{proof}
The distributional convergence in the equations passes the momentum and
conformation balances to the limit.  The strong convergence of \(B_n\) preserves
the positive-cone parametrization, because \(A_n=e^{B_n}\to e^B=A\).  Hence
\((u,p,B,R,S)\) is a positive-cone Reynolds state.

It remains only to pass the sign of the augmented budget.  For each \(n\),
admissibility means
\[
  P_n\,dt+\frac{\alpha}{2}
  \left[
  \left(\int_{\T^d}G_n:S_n\,dx\right)dt+\mu_n
  \right]\le0
  \qquad\hbox{in }\mathcal M(E).
\]
By \eqref{eq:closed-framework-P}--\eqref{eq:closed-framework-product}, the
left-hand side converges weakly-* to
\(\mathfrak B[u,B,R,S,\mu]\).  The cone of non-positive finite measures is
weakly-* closed: testing against any non-negative \(\varphi\in C_c(E)\) and
passing to the limit preserves the inequality.  Therefore
\(\mathfrak B[u,B,R,S,\mu]\le0\), which is the claimed augmented admissibility.
\end{proof}

\begin{definition}[Compactness realization by weak solutions]
\label{def:weak-compactness-realization}
Let \((u,A=e^B)\) be a smooth positive-cone profile on a time window \(E\).
We say that \((u,A)\) is compactly realized by admissible weak solutions with
defects \((R,S,\mu)\) if there are energy--entropy admissible Oldroyd--B weak
solutions \((u_n,A_n)\) and pressures \(p_n\) such that, after passing to a
subsequence, the following hold.

First, the barycentric fields converge to the target profile in the sense
needed to pass the linear terms in the equations.  Second, the nonlinear weak
limits generate the distributional defects \(R\) and \(S\):
\begin{align}
  \partial_t u+\divv(u\otimes u)-\nu\Delta u+\nabla p
  &=
  \alpha\divv(A-\Id)+\divv R,
  \label{eq:weak-compactness-momentum}\\
  \partial_t A+u\cdot\nabla A-\nabla u\,A-A(\nabla u)^T
  +\lambda^{-1}(A-\Id)
  &=
  S.
  \label{eq:weak-compactness-conformation}
\end{align}
Third, with \(G=\Id-A^{-1}\) and with \(P\) the pressure-free work of \(R\),
the product measures associated with the realizing sequence converge to
\[
  \left(\int_{\T^d}G:S\,dx\right)dt+\mu
\]
in the sense of Theorem~\ref{thm:defect-compactification}.  The signed measure
\(\mu\) is called the weak-solution product defect.  Equivalently, the realizing
sequence induces the augmented state \((u,p,B,R,S,\mu)\).
\end{definition}

\subsubsection{Preparatory cone and oscillatory tools for packet lifting}
\label{subsec:packet-tools}

The preceding definition leaves the packet-lifting property as a genuine PDE
input.  The next three lemmas record the algebraic and oscillatory tools that
are already available before such a lifting theorem is attempted.  They do not
construct weak solutions.  Rather, they specify the positive-cone dictionary,
the principal averaging rule, and the anti-divergence gain that any localized
positive-cone lifting argument must satisfy.

\begin{lemma}[Square-root positive-cone dictionary]
\label{lem:sqrt-cone-dictionary}
Let \(K\) be a compact subset of a smooth space-time chart, and let
\(H\) be the restriction to \(K\) of a \(C^s\) map into \(\Sym^d_+\), with
\(s\ge1\).  For every \(\eta>0\) there are finitely many unit vectors
\(\xi_\nu\in\mathbb R^d\), rank-one tensors
\[
  Q_\nu=\xi_\nu\otimes\xi_\nu,
  \qquad \nu=1,\ldots,N,
\]
and non-negative functions \(a_\nu\in C^s(K)\) such that
\begin{equation}
  H+\eta I=\sum_{\nu=1}^N a_\nu^2 Q_\nu
  \qquad\hbox{on }K .
  \label{eq:sqrt-cone-dictionary}
\end{equation}
Consequently
\[
  \left\|H-\sum_{\nu=1}^N a_\nu^2Q_\nu\right\|_{C^0(K)}
  \le C_d\eta .
\]
Moreover, for \(0\le r\le s\),
\begin{equation}
  \sum_{\nu=1}^N \norm{a_\nu}_{C^r(K)}
  \le
  C_{d,s,r,\eta,M}\bigl(1+\norm{H}_{C^r(K)}^{\gamma_r}\bigr),
  \qquad
  M=\norm{H}_{C^0(K)} ,
  \label{eq:sqrt-cone-derivative-bound}
\end{equation}
for constants depending only on the displayed parameters and on the finite
cone cover.
\end{lemma}

\begin{proof}
Set \(H_\eta=H+\eta I\).  The range of \(H_\eta\) is contained in the compact
subset
\[
  \mathcal K_{\eta,M}
  =
  \{M'\in\Sym^d:\eta I\le M'\le(M+\eta)I\}
\]
of the positive definite cone.  We first build a positive rank-one frame near
one point \(M_0\in\mathcal K_{\eta,M}\).  After an orthogonal change of basis,
write
\[
  M_0=\operatorname{diag}(\lambda_1,\ldots,\lambda_d),
  \qquad \lambda_i>0.
\]
For \(i<j\) put \(q_{ij}^\pm=e_i\pm e_j\).  Choose \(\beta>0\) so small that
\(\lambda_i-2(d-1)\beta>0\) for every \(i\).  In a neighbourhood of \(M_0\),
every symmetric matrix \(M=(M_{ij})\) has the representation
\[
  M=
  \sum_i c_i(M)e_i\otimes e_i
  +\sum_{i<j}c_{ij}^+(M)q_{ij}^+\otimes q_{ij}^+
  +\sum_{i<j}c_{ij}^-(M)q_{ij}^-\otimes q_{ij}^- ,
\]
where
\[
  c_{ij}^\pm(M)=\beta\pm \frac12 M_{ij},
  \qquad
  c_i(M)=M_{ii}-\sum_{j\ne i}\bigl(c_{ij}^+(M)+c_{ij}^-(M)\bigr).
\]
For \(M\) close enough to \(M_0\), all these coefficients are strictly
positive.  Undoing the orthogonal change of basis gives a local representation
by fixed rank-one tensors with positive affine coefficients.

By compactness of \(\mathcal K_{\eta,M}\), finitely many such neighbourhoods
\(U_j\) cover the range.  On \(U_j\) write
\[
  M=\sum_{\mu=1}^{m_j}c_{j\mu}(M)Q_{j\mu},
  \qquad c_{j\mu}(M)>0 .
\]
After shrinking the \(U_j\), assume \(c_{j\mu}\ge\kappa_j>0\) on \(U_j\).
Choose a smooth squared partition of unity on a neighbourhood of
\(\mathcal K_{\eta,M}\):
\[
  \theta_j\ge0,\qquad \operatorname{supp}\theta_j\subset U_j,
  \qquad \sum_j\theta_j^2=1 .
\]
Then, for \(M\in\mathcal K_{\eta,M}\),
\[
  M=
  \sum_j\sum_{\mu=1}^{m_j}
  \theta_j(M)^2 c_{j\mu}(M)Q_{j\mu}.
\]
Taking \(M=H_\eta(t,x)\) and setting
\[
  a_{j\mu}(t,x)
  =
  \theta_j(H_\eta(t,x))\,c_{j\mu}(H_\eta(t,x))^{1/2}
\]
gives \eqref{eq:sqrt-cone-dictionary}.  The square roots are \(C^s\) because
the coefficients are bounded below on the supports of the \(\theta_j\).
The derivative estimate \eqref{eq:sqrt-cone-derivative-bound} follows from
the chain rule applied to the finitely many smooth coefficient functions.
\end{proof}

\begin{remark}[Cone floor]
\label{rem:cone-floor}
The term \(\eta I\) is deliberate.  It avoids taking square roots at the
singular tip of the closed positive cone.  In a future iteration this
deterministic floor must be charged to the next residual budget; the lemma
does not claim exact smooth square-root coordinates on changing-rank targets.
\end{remark}

\begin{lemma}[Oscillatory products and principal averages]
\label{lem:oscillatory-product-averaging}
Let \(k,\ell\in\mathbb Z^d\setminus\{0\}\), let \(\lambda\in\mathbb N\), and
write
\[
  A_r(a)=\sum_{|\beta|\le r}\norm{\partial^\beta a}_{C^0(\mathbb T^d)} .
\]
For every \(r\ge0\),
\begin{equation}
  \norm{a(x)e^{i\lambda k\cdot x}}_{C^r}
  \le
  C_{r,k}\sum_{j=0}^r
  \lambda^{r-j}A_j(a).
  \label{eq:osc-product-Cr}
\end{equation}
For every \(m\ge1\),
\begin{equation}
  \norm{a(x)e^{i\lambda k\cdot x}}_{C^{-m}}
  \le C_{m,k}\lambda^{-m}A_m(a),
  \label{eq:osc-product-negative}
\end{equation}
where \(C^{-m}\) denotes the dual norm against \(C^m\) test functions.  If
\(k+\ell\ne0\), then
\begin{equation}
  \norm{a(x)b(x)e^{i\lambda(k+\ell)\cdot x}}_{C^{-m}}
  \le C_{m,k,\ell}\lambda^{-m}A_m(ab).
  \label{eq:nonresonant-product}
\end{equation}
If \(\ell=-k\), the paired interaction has a non-oscillatory principal
average.  For instance,
\begin{equation}
  2a\cos(\lambda k\cdot x)\,b\cos(\lambda k\cdot x)
  =
  ab+ab\cos(2\lambda k\cdot x),
  \label{eq:paired-cosine-average}
\end{equation}
and the second term is \(O(\lambda^{-m}A_m(ab))\) in \(C^{-m}\).
\end{lemma}

\begin{proof}
The \(C^r\) bound is Leibniz' rule: each derivative landing on the exponential
contributes a factor \(\lambda k\), while derivatives landing on the amplitude
are recorded in \(A_j(a)\).  For \eqref{eq:osc-product-negative}, use the
integration-by-parts operator
\[
  L_k=\frac{k\cdot\nabla}{i\lambda |k|^2},
  \qquad
  L_ke^{i\lambda k\cdot x}=e^{i\lambda k\cdot x}.
\]
Moving \(L_k\) onto \(a\varphi\) \(m\) times gains \(\lambda^{-m}\) when
testing against \(\norm{\varphi}_{C^m}\le1\).  The non-resonant estimate is the
same argument with \(k+\ell\) in place of \(k\).  The resonant identity
\eqref{eq:paired-cosine-average} is the elementary trigonometric product
formula, followed again by \eqref{eq:osc-product-negative}.
\end{proof}

\begin{lemma}[High-frequency anti-divergence gain]
\label{lem:anti-divergence-gain}
Let \(\mathcal A\) be a fixed symmetric trace-free anti-divergence operator on
\(\mathbb T^d\), so that
\[
  \operatorname{div}\mathcal A f=f-\langle f\rangle
\]
for smooth vector fields \(f\).  Let
\[
  f_\lambda(x)=a(x)e^{i\lambda k\cdot x},
  \qquad k\ne0,
\]
where \(a\) is vector-valued and smooth, and set
\[
  \widetilde f_\lambda=f_\lambda-\langle f_\lambda\rangle .
\]
Then, for every \(r\ge0\),
\begin{equation}
  \norm{\mathcal A\widetilde f_\lambda}_{C^r}
  \le
  C_{r,k}\sum_{j=0}^{r+1}
  \lambda^{r-1-j}A_j(a),
  \label{eq:anti-divergence-Cr}
\end{equation}
and, for every \(m\ge0\),
\begin{equation}
  \norm{\mathcal A\widetilde f_\lambda}_{C^{-m}}
  \le C_{m,k}\lambda^{-m-1}A_{m+1}(a).
  \label{eq:anti-divergence-negative}
\end{equation}
The subtracted mean obeys
\begin{equation}
  |\langle f_\lambda\rangle|
  \le C_{m,k}\lambda^{-m}A_m(a).
  \label{eq:localized-mean-bound}
\end{equation}
\end{lemma}

\begin{proof}
On nonzero Fourier modes, \(\mathcal A\) is a multiplier of order \(-1\).
Applied to \(a e^{i\lambda k\cdot x}\), its leading symbolic term is
\[
  e^{i\lambda k\cdot x}\lambda^{-1}B_k(a),
\]
with \(B_k\) a fixed linear map depending on the direction \(k\).  Commutators
in which derivatives fall on \(a\) carry further inverse powers of \(\lambda\).
Differentiating the resulting expansion gives
\eqref{eq:anti-divergence-Cr}.  Pairing the same expansion against \(C^m\)
test functions and applying Lemma~\ref{lem:oscillatory-product-averaging}
gives \eqref{eq:anti-divergence-negative}.  Finally,
\eqref{eq:localized-mean-bound} is \eqref{eq:osc-product-negative} tested
against the constant function \(1\).
\end{proof}

\begin{remark}[Scale gates for a lifting theorem]
\label{rem:packet-scale-gates}
Suppose a future packet-lifting step receives the cone decomposition
\[
  H_q+\eta_q I=\sum_{\nu=1}^{N_q}a_{q,\nu}^2Q_{q,\nu}
\]
from Lemma~\ref{lem:sqrt-cone-dictionary}.  The non-resonant products in a
phase-coloured packet family must satisfy a gate of the form
\[
  \lambda_q^{-m}
  \sum_{\nu,\mu}A_m(a_{q,\nu}a_{q,\mu})
  \ll \delta_{q+1},
\]
and the momentum anti-divergence correction must satisfy
\[
  \lambda_q^{-1}\sum_\nu A_1(a_{q,\nu})
  +\lambda_q^{-2}\sum_\nu A_2(a_{q,\nu})
  +\cdots
  \ll \delta_{q+1}.
\]
If the packets are localized by cutoffs \(\chi_q\), the same inequalities must
hold with \(a_{q,\nu}\) replaced by \(\chi_q a_{q,\nu}\), and the mean
subtraction bound \eqref{eq:localized-mean-bound} must also be charged to the
residual budget.  These conditions are necessary bookkeeping for a localized
lifting theorem; they are not, by themselves, such a theorem.
\end{remark}

\begin{definition}[Admissible packet-lifting topology]
\label{def:admissible-lifting-topology}
Fix a packet sequence \((P_n,G_n,S_n)\) on \(E\times\T^d\) and a candidate
induced residual-data sequence
\((\widehat P_n,\widehat G_n,\widehat S_n)\).  A convergence criterion or
topology \(\mathfrak T\) for the residual-data error is called admissible if
convergence of
\[
  (\widehat P_n-P_n,\widehat G_n-G_n,\widehat S_n-S_n)\to0
  \qquad\hbox{in }\mathfrak T
\]
implies the three consequences used by the compactness argument:
\begin{align}
  (\widehat P_n-P_n)\,dt
  &\stackrel{*}{\rightharpoonup}0
  \quad\hbox{in }\mathcal M(E), \label{eq:packet-lift-P}\\
  \widehat G_n-G_n&\to0,\qquad
  \widehat S_n-S_n\to0
  \quad\hbox{in }\mathcal D'(E\times\T^d;\Sym^d),
  \label{eq:packet-lift-GS}
\end{align}
and the product-error measures are continuous at the packet sequence:
\begin{equation}
  \left[
  \int_{\T^d}\widehat G_n:\widehat S_n\,dx
  -\int_{\T^d}G_n:S_n\,dx
  \right]dt
  \stackrel{*}{\rightharpoonup}0
  \quad\hbox{in }\mathcal M(E).
  \label{eq:packet-lift-product}
\end{equation}
Thus admissibility is a contract on the quantities that enter the weak
formulation and the augmented budget; it is not tied to a particular norm.
\end{definition}

\begin{definition}[Positive-cone weak-solution lifting of residual packets]
\label{def:packet-lifting}
Let \((u,p,B,R,S,\mu)\) be a smooth augmented positive-cone Reynolds state with
\(\mu=m(t)\,dt\), and let \((P_n,G_n,S_n)\) be a positive-cone residual packet
sequence produced by Theorem~\ref{thm:strict-augmented-converse}, with the
localized bounds of Corollary~\ref{cor:quantitative-localized-packets} whenever
localization is required.  The preparatory estimates in
Lemmas~\ref{lem:sqrt-cone-dictionary}--\ref{lem:anti-divergence-gain} describe
the algebraic cone decomposition and the principal oscillatory gains expected
of any localized implementation.  We say that this packet sequence admits a
positive-cone weak-solution lifting if there are energy--entropy admissible
Oldroyd--B weak solutions \((u_n,A_n)\), pressures \(p_n\), and induced residual
data
\[
  (P_n^{\rm sol},G_n^{\rm sol},S_n^{\rm sol})
\]
from their compactness defects such that the barycentric fields converge to
\((u,A=e^B)\), the distributional defects converge to \(R,S\), and for some
admissible packet-lifting topology \(\mathfrak T\) in the sense of
Definition~\ref{def:admissible-lifting-topology},
\[
  (P_n^{\rm sol}-P_n,G_n^{\rm sol}-G_n,S_n^{\rm sol}-S_n)\to0
  \qquad\hbox{in }\mathfrak T .
\]
\end{definition}

\begin{remark}[Concrete sufficient lifting topologies]
\label{rem:concrete-lifting-topologies}
The concrete topology
\[
  P_n^{\rm sol}-P_n\to0\hbox{ in }L^1(E),\qquad
  G_n^{\rm sol}-G_n\rightharpoonup0,\quad
  S_n^{\rm sol}-S_n\rightharpoonup0
  \hbox{ weakly in }L^2(E\times\T^d),
\]
together with \eqref{eq:packet-lift-product}, is one sufficient certificate.
It is not part of the definition.  Strong Sobolev convergence, compactness plus
equiintegrability, biting/Young-measure convergence, or a localized
distributional topology are equally admissible if they imply
\eqref{eq:packet-lift-P}--\eqref{eq:packet-lift-product} and the weak
formulation can be passed to the limit.
\end{remark}

\begin{theorem}[Residual-error certificate for packet lifting]
\label{thm:strong-error-lifting-certificate}
Let \((P_n,G_n,S_n)\) be a residual packet sequence on \(E\times\T^d\) with
\[
  \sup_n\bigl(\norm{G_n}_{L^2(E\times\T^d)}
  +\norm{S_n}_{L^2(E\times\T^d)}\bigr)<\infty .
\]
Let \((\widehat P_n,\widehat G_n,\widehat S_n)\) be induced residual data from
a candidate weak-solution approximation, and set
\[
  e^P_n=\widehat P_n-P_n,\qquad
  e^G_n=\widehat G_n-G_n,\qquad
  e^S_n=\widehat S_n-S_n .
\]
If
\begin{equation}
  e^P_n\to0\hbox{ in }L^1(E),\qquad
  e^G_n\to0,\quad e^S_n\to0
  \hbox{ in }L^2(E\times\T^d),
  \label{eq:strong-error-certificate}
\end{equation}
then the residual-data error defines an admissible packet-lifting topology in
the sense of Definition~\ref{def:admissible-lifting-topology}.  Consequently,
any energy--entropy admissible weak solutions producing
\((\widehat P_n,\widehat G_n,\widehat S_n)\) and satisfying
\eqref{eq:strong-error-certificate} give a positive-cone weak-solution lifting
of the packet sequence.

The same conclusion holds under the following one-strong-factor variant:
\[
  e^P_n\to0\hbox{ in }L^1(E),\qquad
  e^G_n\to0\hbox{ in }L^2(E\times\T^d),\qquad
  e^S_n\rightharpoonup0\hbox{ weakly in }L^2(E\times\T^d),
\]
provided the modulated product error satisfies
\begin{equation}
  \left(\int_{\T^d}G_n:e^S_n\,dx\right)dt
  \stackrel{*}{\rightharpoonup}0
  \qquad\hbox{in }\mathcal M(E).
  \label{eq:modulated-weak-annihilation}
\end{equation}
The symmetric statement with \(G\) and \(S\) interchanged is also valid.
\end{theorem}

\begin{proof}
The convergence of \(e^P_n\) in \(L^1(E)\) gives
\[
  e^P_n\,dt\stackrel{*}{\rightharpoonup}0
  \qquad\hbox{in }\mathcal M(E).
\]
The \(L^2\) convergence in \eqref{eq:strong-error-certificate} implies
distributional convergence of \(e^G_n\) and \(e^S_n\) to zero.  It remains to
check the product measure.  Since
\[
  \widehat G_n:\widehat S_n-G_n:S_n
  =
  e^G_n:S_n+G_n:e^S_n+e^G_n:e^S_n,
\]
Cauchy's inequality gives the total-variation estimate
\[
\begin{aligned}
  &\left\|
  \left[
  \int_{\T^d}\widehat G_n:\widehat S_n\,dx
  -\int_{\T^d}G_n:S_n\,dx
  \right]dt
  \right\|_{\mathcal M(E)}
\\
  &\qquad\le
  \norm{e^G_n}_{L^2}\norm{S_n}_{L^2}
  +\norm{G_n}_{L^2}\norm{e^S_n}_{L^2}
  +\norm{e^G_n}_{L^2}\norm{e^S_n}_{L^2}.
\end{aligned}
\]
The right-hand side tends to zero by the assumed \(L^2\) convergence and the
uniform packet bounds.  Thus \eqref{eq:packet-lift-P}--\eqref{eq:packet-lift-product}
hold, so the topology is admissible.

For the one-strong-factor variant, the first and third terms above still tend
to zero because \(e^G_n\to0\) strongly and \(S_n\), \(e^S_n\) are bounded in
\(L^2\).  The middle term is exactly
\eqref{eq:modulated-weak-annihilation}.  Distributional convergence of
\(e^S_n\) follows from its weak \(L^2\) convergence.  The symmetric case is the
same argument with the roles of the two factors exchanged.
\end{proof}

\begin{remark}[Certificate contract for packet lifting]
\label{rem:lifting-certificate-contract}
Definition~\ref{def:packet-lifting} is intended as a falsifiable certificate,
not as a hidden existence assertion.  A proof of the packet-lifting property
must specify the exact Oldroyd--B model, domain, boundary conditions, pressure
gauge, and approximation scheme; prove solvability of the approximate problems
with uniform energy--entropy bounds; pass the transport, stretching,
convection, and stress-divergence products in an admissible topology strong
enough for the weak formulation; justify the positive-cone and matrix-entropy
statements through measurable spectral calculus or a regularized log-entropy
argument; and state whether each residual, Reynolds stress, oscillation defect, or
concentration defect vanishes or remains in the augmented solution concept.  In
particular, an exact weak-solution converse requires a defect-vanishing or
selection mechanism in addition to the residual packet construction.
\end{remark}

\begin{theorem}[Conditional weak-solution converse from packet lifting]
\label{thm:conditional-full-converse}
Let \((u,p,B,R,S,\mu)\) be a smooth augmented positive-cone Reynolds state on
\(E\), with \(\mu=m(t)\,dt\).  Assume that \(I-G\ge\kappa I\) and that the
augmented budget has a strict negative margin.  Let
\((P_n,G_n,S_n)\) be any localized positive-cone packet sequence supplied by
Theorem~\ref{thm:strict-augmented-converse}, with the quantitative bounds of
Corollary~\ref{cor:quantitative-localized-packets}.  If this packet sequence
admits a positive-cone weak-solution lifting for some admissible
packet-lifting topology in the sense of
Definitions~\ref{def:admissible-lifting-topology} and
\ref{def:packet-lifting}, then \((u,A=e^B)\) is compactly realized
by energy--entropy admissible Oldroyd--B weak solutions with defects
\((R,S,\mu)\).  Consequently, the strict augmented converse becomes a
weak-solution converse only for classes of states for which the packet-lifting
property has been proved.
\end{theorem}

\begin{proof}
The packet sequence from Theorem~\ref{thm:strict-augmented-converse} satisfies
\[
  P_n=P,\qquad
  G_n\rightharpoonup G,\qquad
  S_n\rightharpoonup S,
\]
and
\[
  \left(\int_{\T^d}G_n:S_n\,dx\right)dt
  \stackrel{*}{\rightharpoonup}
  \left(\int_{\T^d}G:S\,dx\right)dt+\mu .
\]
The admissible-topology consequences
\eqref{eq:packet-lift-P}--\eqref{eq:packet-lift-product} say that the residual
data induced by the weak solutions have the same distributional limits and the
same product defect as the packet sequence.  The barycentric convergence and
the convergence of distributional defects are part of
Definition~\ref{def:packet-lifting}.  Therefore the weak solutions compactly
realize \((u,A=e^B)\) with defects \((R,S,\mu)\) in the sense of
Definition~\ref{def:weak-compactness-realization}.
\end{proof}

\begin{theorem}[Necessary compactness law for weak-solution realizations]
\label{thm:weak-compactness-realization}
Let a smooth positive-cone profile \((u,A=e^B)\) be compactly realized by
energy--entropy admissible Oldroyd--B weak solutions in the sense of
Definition~\ref{def:weak-compactness-realization}.  Then the induced
\((u,p,B,R,S,\mu)\) is an admissible augmented positive-cone Reynolds state:
\begin{equation}
  P\,dt+\frac{\alpha}{2}
  \left(\int_{\T^d}G:S\,dx\right)dt
  +\frac{\alpha}{2}\mu\le0
  \qquad\hbox{in }\mathcal M(E).
  \label{eq:weak-compactness-budget}
\end{equation}
Consequently, if \(\mu=0\), every compact weak-solution realizing sequence satisfies
the residual-work admissibility inequality of
Theorem~\ref{thm:closed-realization}.  If a smooth target profile violates the
augmented inequality \eqref{eq:weak-compactness-budget} for every possible
product defect allowed by the chosen compactness class, then it cannot be
obtained from that class of admissible weak solutions.
\end{theorem}

\begin{proof}
The distributional limits in
\eqref{eq:weak-compactness-momentum}--\eqref{eq:weak-compactness-conformation}
are exactly the equations in Definition~\ref{def:weak-compactness-realization}.
Since \(A=e^B\), the limiting state remains in the positive cone, and therefore
\((u,p,B,R,S,\mu)\) is an augmented positive-cone Reynolds state in the sense of
Definition~\ref{def:augmented-reynolds-state}, with the residuals \(R\) and
\(S\) generated by the weak compactness defects.

For the realizing weak solutions, energy--entropy admissibility gives the
non-positive residual-work measure at the approximate level after the same
pressure-free reduction used in Proposition~\ref{prop:canonical-pressure-free}.
By assumption, the pressure-free work converges to \(P\) and the signed product
converges with defect \(\mu\).  The compactification theorem,
Theorem~\ref{thm:defect-compactification}, therefore gives
\eqref{eq:weak-compactness-budget}.  If \(\mu=0\), the measure inequality is
absolutely continuous with density
\[
  P(t)+\frac{\alpha}{2}\int_{\T^d}G(t):S(t)\,dx,
\]
and hence recovers the residual-work admissibility inequality.  The final
claim is the contrapositive of \eqref{eq:weak-compactness-budget}.
\end{proof}

\begin{theorem}[Packet-ledger obstruction to automatic zero-defect closure]
\label{thm:packet-ledger-obstruction}
Let \(\mathcal Y\) be a Hausdorff topological vector space in which summable
series and limits are meaningful.  Suppose a positive-cone packet-lifting
scheme carries a coarse conformation residual \(Q_q\in\mathcal Y\), an
accumulated weak stretching defect \(D_q\in\mathcal Y\), and intended packet
averages \(H_q\in\mathcal Y\), with update laws
\begin{equation}
  Q_{q+1}=Q_q-H_q+E_q,
  \qquad
  D_{q+1}=D_q+H_q+F_q .
  \label{eq:packet-ledger-update}
\end{equation}
Assume that the error series
\[
  \sum_{q=0}^{\infty}E_q,
  \qquad
  \sum_{q=0}^{\infty}F_q
\]
converge in \(\mathcal Y\), and that
\[
  Q_q\to0,\qquad D_q\to D_\infty
\]
in \(\mathcal Y\).  Then
\begin{equation}
  D_\infty
  =
  Q_0+D_0+\sum_{q=0}^{\infty}(E_q+F_q).
  \label{eq:packet-ledger-limit}
\end{equation}
In particular, when \(D_0=0\), zero-defect closure \(D_\infty=0\) requires an
exact cancellation of the initial coarse residual \(Q_0\) by the summable
non-principal errors.  A stretching-packet cancellation of a nonzero \(Q_0\)
therefore does not by itself select a standard weak-solution limit.
\end{theorem}

\begin{proof}
Adding the two identities in \eqref{eq:packet-ledger-update} gives
\[
  Q_{q+1}+D_{q+1}
  =
  Q_q+D_q+E_q+F_q .
\]
Iterating from \(q=0\) to \(q=N-1\) yields
\[
  Q_N+D_N
  =
  Q_0+D_0+\sum_{q=0}^{N-1}(E_q+F_q).
\]
The assumed convergence of \(Q_N\), \(D_N\), and the error series gives
\eqref{eq:packet-ledger-limit}.  The final assertion follows by setting
\(D_0=0\).  The principal average \(H_q\) cancels from the sum because it is
subtracted from the smooth residual and added to the weak stretching ledger.
\end{proof}

\begin{remark}[Meaning of the ledger obstruction]
Theorem~\ref{thm:packet-ledger-obstruction} does not rule out standard weak
solutions.  It says only that a lifting mechanism based solely on high-high
stretching averages cannot both cancel a generic nonzero coarse conformation
residual and make the associated weak stretching defect disappear.  A
zero-defect theorem would need a genuinely additional selection mechanism:
for example a non-stretching conformation corrector, a commutator regularity
threshold, or an entropy principle strong enough to exclude the defect.
\end{remark}

\begin{remark}[The lifting interface]
Theorem~\ref{thm:closed-compactness-framework} gives a closed compactness
framework, and Theorem~\ref{thm:weak-compactness-realization} places genuine
weak-solution realizing sequences inside it.  Together with
Theorem~\ref{thm:strict-augmented-converse}, this proves the converse at the
residual-data level and identifies the PDE input needed beyond residual
compactness.  Theorem~\ref{thm:strong-error-lifting-certificate} turns that PDE
input into a concrete error target: strong residual-data matching, or one
strong factor plus modulated weak annihilation of the other, is enough to
certify the packet lift.  A coupled weak-solution realization must therefore
produce exact Oldroyd--B weak solutions whose induced residual data meet this
certificate while preserving \(A>0\), the energy--entropy inequality, and, if
one wants a standard rather than defect-augmented weak limit, a zero-defect
selection condition not supplied by the principal packet average alone.  The
preparatory tools in Lemmas~\ref{lem:sqrt-cone-dictionary}--\ref{lem:anti-divergence-gain}
make two parts of this lifting problem explicit: the residual target must first
be written as square-root-compatible rank-one positive-cone atoms, and all
non-principal packet interactions must be charged through oscillatory averaging
or anti-divergence gains.
\end{remark}

\section{Discussion: Residual Work, Compactness, and Numerical Diagnostics}
\label{sec:discussion}

The residual-work law reduces admissibility to one signed ledger.  After
pressure and mean gauges have been removed, the momentum residual contributes
the scalar work \(P\), while the conformation residual contributes only through
\(\int G:S\).  This gives a sharp cone at fixed time, the windowed estimate
\(W_E\le(\alpha/2)\eta_E L_ES_E\), and the least-cost Hilbert projection that
repairs a violating residual.  The same geometry also explains the weak
compactness threshold: if both \(G_n\) and \(S_n\) converge only weakly, the
product \(G_n:S_n\) may carry an additional measure \(\mu\), and retaining
\(\mu\) is exactly what closes the residual ledger.

For numerical and closure applications, the useful object is the defect gap
\[
  \mathfrak D_E(P,G,S)
  =
  \int_E\left[
  P(t)+\frac{\alpha}{2}\int_{\mathbb T^d}G(t):S(t)\,dx
  \right]_+dt .
\]
Proposition~\ref{prop:error-bar-defect-gap} makes this gap stable under
\(L^1\times L^2\times L^2\) data errors, while
Corollary~\ref{cor:conservative-channel-certificate} gives an alignment-free
certificate when only upper bounds for \(L_E\), \(S_E\), and \(\eta_E\) are
available.  Proposition~\ref{prop:concrete-layer} supplies a manufactured
calibration case with known scaling; filtered DNS or coarse-grid data enter the
same diagnostic by replacing the explicit ansatz with reconstructed residuals.

The local positive-cone correctors should be read as compatibility tools.  They
show that the optimal paying direction is not excluded by the positive-cone
geometry of the conformation variable.  Global realization is then governed by
the compactness topology: strong residual-data convergence preserves the
unaugmented budget, whereas weak packet limits charge the missing product to
the defect measure.  This is the separation between local repair geometry and
global residual admissibility.

\begin{proposition}[Necessary signed component of any corrector]
\label{prop:necessary-corrector}
Let \(P(t)>0\) and \(G(t)\not\equiv0\).  If a correction
\(\widetilde S\) makes the pressure-free work admissible, then its projection
onto \(-G\) must satisfy
\[
  \left\langle \widetilde S,
  -\frac{G}{\norm{G}_{L^2}}\right\rangle_{L^2}
  \ge
  \frac{2P(t)}{\alpha\norm{G}_{L^2}} .
\]
In particular, a correction that is \(L^2\)-orthogonal to \(G\) cannot pay
positive pressure-free work.
\end{proposition}

\begin{proof}
The admissibility inequality requires
\[
  -\int_{\T^d}G:\widetilde S\,dx
  \ge \frac{2P(t)}{\alpha}.
\]
Dividing by \(\norm{G}_{L^2}\) gives the projection lower bound.
\end{proof}

\begin{remark}[Open problem: positive-cone weak-solution lifting]
\label{rem:realization-open}
The residual-data converse constructs localized positive-cone packets with a
prescribed augmented product defect.  The open PDE problem is to lift such
packets to energy--entropy admissible Oldroyd--B weak solutions whose induced
momentum Reynolds stress, conformation residual, and \(G:S\) product defect
converge to the same augmented state.  A zero-defect weak-solution limit would
also require a selection mechanism removing the packet-ledger defect identified
in Theorem~\ref{thm:packet-ledger-obstruction}.
\end{remark}

\section{Conclusion}

This paper proves a residual-work compatibility theory for positive-cone
viscoelastic Reynolds states.  The PDE input is the Oldroyd--B/FENE-P
energy--entropy cancellation.  Once incompressible transport, upper-convected
stretching, polymeric stress work, pressure tensors, and spatial means have
been accounted for, the remaining momentum residual has a gauge-invariant
pressure-free work \(P\).  Positive work is admissible only when it is paid by
the entropy-dual conformation channel:
\[
  P(t)+\frac{\alpha}{2}\int_{\T^d}G(t):S(t)\,dx\le0,
  \qquad G=I-A^{-1}.
\]
Equivalently, on each time window,
\[
  W_E\le \frac{\alpha}{2}\eta_E L_ES_E .
\]
This inequality is sharp.  It gives the least-cost paying direction \(-G\), the
explicit Hilbert-space projection that repairs a proposed conformation
residual, and the closure-subspace obstruction obtained by replacing \(G\) with
its resolved projection.

\paragraph{Outputs of the theory.}
The results give a gauge-invariant residual admissibility test for smooth
positive-cone Reynolds profiles, filtered data, reduced closures, and
coarse-grid numerical residuals.  They give an exact metric repair formula, so
the distance from a proposed residual to the admissible class is computable.
They identify the compactness threshold: strong convergence of \(P,G,S\) closes
the signed-work budget, while weak--weak convergence of \(G\) and \(S\) closes
only after retaining the product-defect measure \(\mu\).  They also extend the
ledger to FENE-P through
\[
  G_b(C)=\frac{b-d}{b-\operatorname{tr}C}I-C^{-1},
\]
so the finite-extensibility boundary becomes part of the residual-work
geometry.  The diagnostic workflow and the finite-thickness shear-layer
calibration show how these outputs are evaluated in a numerical postprocessor:
compute \(P\), \(G\), and \(S\); compare the signed gap with the data-error bar;
and, when needed, compute the least-cost repair or the projected closure repair.

\paragraph{Scope and open direction.}
The theorem is a residual-admissibility and compactness result.  Its natural
inputs are residual traces produced by filtering, closure modeling, numerical
reconstruction, or residual-level packet limits.  Exact smooth solutions have
zero residual trace, while a full weak-solution realization of the residual
packets requires an additional PDE construction.  The resulting open problem is
precise: lift the localized positive-cone residual packets to
energy--entropy-admissible Oldroyd--B weak solutions while controlling the
momentum Reynolds stress, the conformation residual, the pressure recovery, and
the \(G:S\) product defect.  If the target is a standard weak limit with no
retained stretching defect, the construction must also select away the
packet-ledger defect identified in Theorem~\ref{thm:packet-ledger-obstruction}.

The main message is a signed-work ledger for positive-cone viscoelastic
residuals.  Pressure can be gauged away, but positive residual work must be
paid by an entropy-dual conformation residual, by a retained product-defect
measure in weak limits, or by a non-vanishing residual-data error in an
approximation of an inadmissible profile.


\begin{thebibliography}{99}

\bibitem{BarrettBoyaval}
J.~W. Barrett and S.~Boyaval.
\newblock Existence and approximation of a regularized Oldroyd--B model.
\newblock \emph{Mathematical Models and Methods in Applied Sciences},
21:1783--1837, 2011.

\bibitem{BealeKatoMajda}
J.~T. Beale, T.~Kato, and A.~Majda.
\newblock Remarks on the breakdown of smooth solutions for the 3-D Euler
equations.
\newblock \emph{Communications in Mathematical Physics}, 94:61--66, 1984.

\bibitem{Bird}
R.~B. Bird, R.~C. Armstrong, and O.~Hassager.
\newblock \emph{Dynamics of Polymeric Liquids. Vol. 1: Fluid Mechanics}.
\newblock Wiley, second edition, 1987.

\bibitem{CheminMasmoudi}
J.-Y. Chemin and N.~Masmoudi.
\newblock About lifespan of regular solutions of equations related to
viscoelastic fluids.
\newblock \emph{SIAM Journal on Mathematical Analysis}, 33:84--112, 2001.

\bibitem{ChilcottRallison1988}
M.~D. Chilcott and J.~M. Rallison.
\newblock Creeping flow of dilute polymer solutions past cylinders and spheres.
\newblock \emph{Journal of Non-Newtonian Fluid Mechanics}, 29:381--432, 1988.

\bibitem{ConstantinKliegl}
P.~Constantin and M.~Kliegl.
\newblock Note on global regularity for two-dimensional Oldroyd--B fluids with
diffusive stress.
\newblock \emph{Archive for Rational Mechanics and Analysis}, 206:725--740,
2012.

\bibitem{DeLellisSzekelyhidi}
C.~De Lellis and L.~Szekelyhidi Jr.
\newblock The Euler equations as a differential inclusion.
\newblock \emph{Annals of Mathematics}, 170:1417--1436, 2009.

\bibitem{ElgindiRousset}
T.~M. Elgindi and F.~Rousset.
\newblock Global regularity for some Oldroyd--B type models.
\newblock \emph{Communications on Pure and Applied Mathematics}, 68:2005--2021,
2015.

\bibitem{FattalKupferman}
R.~Fattal and R.~Kupferman.
\newblock Constitutive laws for the matrix-logarithm of the conformation tensor.
\newblock \emph{Journal of Non-Newtonian Fluid Mechanics}, 123:281--285, 2004.

\bibitem{FattalKupferman2005}
R.~Fattal and R.~Kupferman.
\newblock Time-dependent simulation of viscoelastic flows at high Weissenberg
number using the log-conformation representation.
\newblock \emph{Journal of Non-Newtonian Fluid Mechanics}, 126:23--37, 2005.

\bibitem{GuillopeSaut}
C.~Guillope and J.-C. Saut.
\newblock Existence results for the flow of viscoelastic fluids with a
 differential constitutive law.
\newblock \emph{Nonlinear Analysis}, 15:849--869, 1990.

\bibitem{LeiMasmoudiZhou}
Z.~Lei, N.~Masmoudi, and Y.~Zhou.
\newblock Remarks on the blowup criteria for Oldroyd models.
\newblock \emph{Journal of Differential Equations}, 248:328--341, 2010.

\bibitem{LionsMasmoudi}
P.-L. Lions and N.~Masmoudi.
\newblock Global solutions for some Oldroyd models of non-Newtonian flows.
\newblock \emph{Chinese Annals of Mathematics. Series B}, 21:131--146, 2000.

\bibitem{MasmoudiFENE}
N.~Masmoudi.
\newblock Well-posedness for the FENE dumbbell model of polymeric flows.
\newblock \emph{Communications on Pure and Applied Mathematics}, 61:1685--1714,
2008.

\bibitem{Oldroyd}
J.~G. Oldroyd.
\newblock On the formulation of rheological equations of state.
\newblock \emph{Proceedings of the Royal Society of London. Series A},
200:523--541, 1950.


\bibitem{DiPernaMajda}
R.~J. DiPerna and A.~J. Majda.
\newblock Oscillations and concentrations in weak solutions of the
incompressible fluid equations.
\newblock \emph{Communications on Pure and Applied Mathematics}, 40:301--345,
1987.

\bibitem{FeireislBook}
E.~Feireisl.
\newblock \emph{Dynamics of Viscous Compressible Fluids}.
\newblock Oxford University Press, 2004.

\bibitem{LionsIncompressible}
P.-L. Lions.
\newblock \emph{Mathematical Topics in Fluid Mechanics. Vol. 1: Incompressible
Models}.
\newblock Oxford University Press, 1996.

\bibitem{PengEndpoint2026}
S.~Peng.
\newblock Pressure quotients and endpoint continuation for non-diffusive
viscoelastic flows.
\newblock arXiv:2606.25258v2, 2026.

\bibitem{Renardy}
M.~Renardy.
\newblock \emph{Mathematical Analysis of Viscoelastic Flows}.
\newblock SIAM, 2000.

\bibitem{TuWangWen}
Y.~Z. Tu, Y.~H. Wang, and H.~Y. Wen.
\newblock The Cauchy problem for an inviscid and non-diffusive Oldroyd--B model
in two dimensions.
\newblock \emph{Nonlinear Analysis: Real World Applications}, 79:104100, 2024.

\end{thebibliography}
\end{document}